\documentclass[a4paper]{amsart}
\usepackage[utf8]{inputenc}
\usepackage{amssymb, amsmath, amsfonts, amsthm,amstext}
\usepackage[utf8]{inputenc}
\usepackage{graphicx}
\usepackage{dsfont}
\usepackage{mathrsfs}
\usepackage{bbm}
\usepackage[english]{babel}
\usepackage{dsfont}
\usepackage{tipa}
\usepackage{textcomp}
\usepackage{pifont}
\usepackage{esint}
\usepackage[T1]{fontenc}
\usepackage{lmodern}
\usepackage[utf8]{inputenc}
\usepackage{indentfirst}
\usepackage{booktabs}
\usepackage{latexsym}
\usepackage{mathtools}
\usepackage{siunitx}
\usepackage{microtype}
\usepackage{braket}
\usepackage{subfig}
\usepackage{stmaryrd}
\usepackage{wrapfig}
\usepackage{setspace}
\usepackage{mathrsfs}
\usepackage[normalem]{ulem}
\usepackage{upgreek}
\usepackage{array}
\usepackage{enumerate}
\usepackage{amstext}
\usepackage{latexsym}
\usepackage{blindtext}
\usepackage{colortbl}
\usepackage{color}
\usepackage{pgfplots}
\usepgfplotslibrary{fillbetween}
\pgfplotsset{width=10cm,compat=1.15}

\usepackage{cite}

\usepackage{dsfont}
\newtheorem{thm}{Theorem}[section]
\newtheorem{cor}[thm]{Corollary}
\newtheorem{lem}[thm]{Lemma}

\theoremstyle{definition}
\newtheorem{definition}[thm]{Definition}
\newtheorem{rem}[thm]{Remark}

\numberwithin{equation}{section}

\renewcommand{\leqslant}{\leq}
\renewcommand{\geqslant}{\geq}
\newcommand{\numberset}{\mathbb}
\newcommand{\N}{\numberset{N}}
\newcommand{\R}{\numberset{R}}
\newcommand{\snr}[1]{\lvert #1\rvert}

\usepackage{eqparbox}
\newcommand{\eqmathbox}[2][eq]{\eqmakebox[#1]{$\displaystyle #2$}}

\def\loc{\operatorname{loc}}
\usepackage{color}
\definecolor{citation}{rgb}{0.2,0.58,0.2} 
\definecolor{formula}{rgb}{0.1,0.2,0.6}
\definecolor{url}{rgb}{0.3,0,0.5} 
\usepackage[colorlinks=true,linkcolor=formula,urlcolor=url,citecolor=citation]{hyperref} 
\def\vu{{\Tilde{u}_\varepsilon}}
\def\vl{{\ell_\omega}}
\def\vL{{\mathcal{L}}}
\def\vm{\mu} 
\def\v{{\varsigma}}
\def\ir{{^{\iota}_{\rho}}}
\def\os{{_{\omega,\sigma}}}
\def\ios{{^{\iota}_{\omega,\sigma}}}
\def\is{{^{\iota}_{\sigma}}}
\def\ei{{E^{\iota}}}

\def\oe{{_{\omega,\varepsilon}}}
\def\ose{{_{\omega,\sigma_{\varepsilon}}}}
\newcommand{\nr}[1]{\lVert #1 \rVert}

\definecolor{ashgrey}{rgb}{0.86, 0.86, 0.86}
\definecolor{lightgray}{rgb}{0.75, 0.75, 0.75}
\definecolor{gray}{rgb}{0.5, 0.5, 0.5}

\title[Nearly linear multi-phase problems]{Regularity for multi-phase problems \\ at nearly linear growth}
\author[De Filippis]{Filomena De Filippis}  
\author[Piccinini]{Mirco Piccinini}  \address{Filomena De Filippis\\Dipartimento di Scienze Matematiche, Fisiche e Informatiche, Universit\`a di Parma\\ Parco Area delle Scienze 53/a, Campus, 43124 Parma, Italy}
\email{\url{filomena.defilippis@unipr.it}}
\address{Mirco Piccinini\\Dipartimento di Scienze Matematiche, Fisiche e Informatiche, Universit\`a di Parma\\ Parco Area delle Scienze 53/a, Campus, 43124 Parma, Italy}
\email{\url{mirco.piccinini@unipr.it}}

\begin{document}

	

	\thanks{{\it Aknowledgements.}
 The first author is supported by INdAM - GNAMPA project ``Regolarita' per problemi ellittici e parabolici con crescite non standard'', CUP\_E53C22001930001
     The second author is supported by INdAM project  ``Problemi non locali: teoria cinetica e non uniforme ellitticit\`a'', CUP\_E53C220019320001 and by the Project ``Local vs Nonlocal: mixed type operators and nonuniform ellipticity", CUP\_D91B21005370003.}

\begin{abstract}	 
	Minima of the log-multiphase variational integral
 \[
 w \mapsto \int_{\Omega} \left[|Dw|\log(1+|Dw|) + a(x)|Dw|^q + b(x)|Dw|^s\right] \, {\rm d}x\,, 
 \]
 have locally H\"older continuous gradient under sharp quantitative bounds linking the growth powers $(q,s)$ to the H\"older exponents of the modulating coefficients $a(\cdot)$ and $b(\cdot)$ respectively.
\end{abstract}

			\maketitle
\vspace{-4mm}
	%\tableofcontents 
\section{Introduction}
In this paper we prove the validity of Schauder theory for a new class of variational integrals at nearly linear growth featuring multiple phases, under optimal assumptions quantifying the rate of nonuniform ellipticity. The model we have in mind is the log-multiphase energy, i.e.:
\begin{flalign} \label{functional}
\begin{array}{c}
\displaystyle
   W^{1,1}_{\loc}(\Omega)\ni w\mapsto    \mathcal{L}(w) \coloneqq \int_{\Omega} H(x,Dw) \, {\rm d}x\\[10pt]\displaystyle H(x,z):=\snr{z}\log(1+\snr{z})+a(x)\snr{z}^{q}+b(x)\snr{z}^{s},
    \end{array}
\end{flalign}
where $\Omega \subset \mathds{R}^n$, $n\ge 2$ is an open set, $1<q,s<\infty$ and $a(\cdot)$, $b(\cdot)$ are nonnegative, bounded coefficients. The crucial feature of integrand $H(\cdot)$, that makes functional $\mathcal{L}(\cdot)$ relevant under the regularity theory viewpoint, relies in the combination of its strong nonuniformly elliptic character with the lack of control over the zero-set of the modulating coefficients $a(\cdot)$, $b(\cdot)$. In fact, following the classification in \cite{dm}, the related pointwise ellipticity ratio
\begin{flalign}\label{per}
\mathcal{R}_{\partial H(\cdot)}(x,z):=\frac{\mbox{highest eigenvalue of }\partial^{2}H(x,z)}{\mbox{lowest eigenvalue of }\partial^{2}H(x,z)}\lesssim 1+\log(1+\snr{z})
\end{flalign}
and the nonlocal one
\begin{eqnarray}\label{ner}
\bar{\mathcal{R}}_{\partial H(\cdot)}(z;B)&:=&\sup_{\substack{\text{$x\in B$}\\ \text{$B\Subset \Omega$ ball}}}\frac{\textnormal{highest eigenvalue of }\partial H(x,z)}{\textnormal{lowest eigenvalue of }\partial H(x,z)}\nonumber \\
&\approx& 1+\nr{a}_{L^{\infty}(B)}\snr{z}^{q-1}+\nr{b}_{L^{\infty}(B)}\snr{z}^{s-1},
\end{eqnarray}
both blow up on large values of the gradient variable. In particular, the behavior at infinity of the nonlocal ellipticity ratio suggests that, rather than chasing competition, each power-type phase interacts only with the superlinear, elliptic one. This aspect is reflected into the possibility of assigning moduli of continuity to $a(\cdot)$ and $b(\cdot)$ without imposing any quantitative relation among them, and no geometric information on the zero level sets of coefficients is required to gain maximal regularity for minima. The main result of this paper is indeed recorded in the following theorem.
\begin{thm} \label{theorem4.1}
Let $u\in W^{1,1}_{\loc}(\Omega)$ be a local minimizer of functional \eqref{functional}, with modulating coefficients $a(\cdot)$, $b(\cdot)$ satisfying
\begin{flalign}\label{abab}
0\le a(\cdot)\in C^{0,\alpha}(\Omega),\qquad \quad 0\le b(\cdot)\in C^{0,\beta}(\Omega),
\end{flalign}
for some $\alpha,\beta\in (0,1]$, and growth exponents $(q,s)$ such that
\begin{equation} \label{q_s}
    1 < q < 1 + \frac{\alpha}{n}, \qquad\quad  \quad  1 < s < 1 + \frac{\beta}{n}
\end{equation}
is verified. Then $Du$ is locally H\"older continuous in $\Omega$. Moreover, given any ball $B_{r}\Subset \Omega$, $r\in (0,1]$, the Lipschitz bound
\begin{align} \label{l_infty_estimate}
  ||Du||_{L^{\infty}(B_{r/2})} \leqslant c \left ( \fint_{B_r} 1+H(x,Du)\, {\rm d}x\right)^{\theta} + c,
\end{align}
holds true with $c \equiv c(n,q,s,||a||_{C^{0,\alpha}}, ||b||_{C^{0,\beta}})$ and $\theta \equiv \theta(n,q,s,\alpha,\beta) > 0$. 
\end{thm}
To keep the amount of technicalities involved in the proof at a reasonable level, Theorem \ref{theorem4.1} is stated for the model example \eqref{functional}, but, after very minor modifications, we are able to handle an arbitrarily large, finite number of phases exhibiting the same phenomenology displayed in \eqref{abab}-\eqref{q_s}.
\begin{cor}\label{maincor}
    Let $k\ge 1$ be an integer, and $u\in W^{1,1}_{\loc}(\Omega)$ be a local minimizer of functional
    \begin{flalign}\label{multif}
    W^{1,1}_{\loc}(\Omega)\ni w\mapsto \int_{\Omega}\snr{Dw}\log(1+\snr{Dw})+\sum_{i=1}^{k}a_{i}(x)\snr{Dw}^{q_{i}}\, {\rm d}x,
    \end{flalign}
    with modulating coefficients $0\le a_{i}(\cdot)\in C^{0,\alpha_{i}}(\Omega)$, for some $\alpha_{i}\in (0,1]$, $i\in\{1,\cdots,k\}$, and growth exponents $q_{i}$ verifying
    \begin{flalign}\label{qi}
    1<q_{i}<1+\frac{\alpha_{i}}{n},
    \end{flalign}
     $\mbox{for all} \ \ i\in \{1,\cdots,k\}$. Then $Du$ is locally H\"older continuous in $\Omega$, and a $W^{1,\infty}$-energy estimate analogous to \eqref{l_infty_estimate} holds true on any ball compactly contained in the domain.
\end{cor}
Let us put our results in the context of the available literature. The log-multiphase integral \eqref{functional} consists in the sum of the classical nearly linear growing energy
\begin{flalign}\label{loglog}
W^{1,1}_{\loc}(\Omega)\ni w\mapsto \int_{\Omega}\snr{Dw}\log(1+\snr{Dw})\, {\rm d}x,
\end{flalign}
studied by Frehse \& Seregin \cite{FrS}, Fuchs \& Seregin \cite{FS1}, Fuchs \& Mingione \cite{FM}, Di Marco \& Marcellini \cite{DiMa}, and several power-type terms, weighted by possibly vanishing coefficients. The terminology "log-multiphase", or "nearly linear growth" accounts for the behavior at infinity of the integrand in \eqref{loglog}, superlinear yet slower than any power larger than one. Functionals \eqref{functional}, \eqref{multif} play a crucial role in the theory of plasticity with logarithmic hardening, a borderline scenario between power hardening \cite{FS1} and perfect plasticity \cite{gme,gm1,gk}. Moreover, under the regularity theory viewpoint, integrals \eqref{functional}, \eqref{multif} can be seen as the limiting configuration as $p\to 1$ of the multi-phase energy
\begin{flalign}\label{mp}
\begin{array}{c}
\displaystyle
W^{1,p}_{\loc}(\Omega)\ni w\mapsto \int_{\Omega}\snr{Dw}^{p}+\sum_{i=1}^{k}a_{i}(x)\snr{Dw}^{q_{i}}\, {\rm d}x\\[10pt]\displaystyle
0\le a_{i}(\cdot)\in C^{0,\alpha_{i}}(\Omega), \ \ \alpha_{i}\in (0,1],\qquad\quad
1<p\le q\le 1+\frac{\alpha_{i}}{n},
\end{array}
\end{flalign}
first treated by De Filippis \& Oh \cite{deoh}, and later on developed in several directions: Calder\'on-Zygmund type estimates \cite{bbo,demine}, maximal regularity in nonhomogeneous function spaces \cite{bb90}, fully nonlinear elliptic equations \cite{dr}, and free transmission problems \cite{df}. The main point in multiphase problems is the presence of several phase transitions, and the corresponding need of carefully tracking the amount of regularity preserved by solutions across the zero level sets $\{x\in \Omega\colon a_{i}(x)=0\}$, that is, when the functional tends to lose part of its ellipticity properties or switch its kind of growth. In turn, multiphase integrals cover the case of multiple phase transitions in double phase problems, whose prototypical model example is given by
\begin{flalign}\label{dp}
\begin{array}{c}
\displaystyle
W^{1,p}_{\loc}(\Omega)\ni w\mapsto \int_{\Omega}\snr{Dw}^{p}+a(x)\snr{Dw}^{q}\, {\rm d}x\\[10pt]\displaystyle
0\le a(\cdot)\in C^{0,\alpha}(\Omega), \ \ \alpha\in (0,1],\qquad\quad
1<p\le q\le 1+\frac{\alpha}{n},
\end{array}
\end{flalign}
modelling the mixture of fluids with different viscosities, first introduced by Zhikov \cite{jko,zh1} in the setting of homogenization or for studying possible occurrences of Lavrentiev phenomenon. A comprehensive regularity theory for double phase problems has been obtained by Baroni \& Colombo \& Mingione in a series of seminal papers \cite{BCM,comi2,comicz,comi1} that generated a huge literature, see \cite{bb901,bb90,bbkkk,bb01,bbo,balci2,balci4,bar1,bar4,boh,dr,df,demine,deoh,HO1,HO2,HO3} and references therein. Let us highlight that the bounds in \eqref{mp}-\eqref{dp}, and our constraints \eqref{q_s}, \eqref{qi}, are by no means technical: these are indeed optimal conditions for regularity, as pointed out by the counterexamples in \cite{balci2,balci4,ELM,FMM}, constructed in the scalar setting. Variational integrals \eqref{mp}-\eqref{dp} are particular instances of softly nonuniformly elliptic problems, characterized by a very low degree on nonuniform ellipticity: in fact it is easy to see that the associated pointwise ellipticity ratio (defined in \eqref{per}) is uniformly bounded, while the nonlocal one \eqref{ner} blows up on high gradients. As a result, the regularity theory for such a class of problems is still perturbative \cite{giagiu,giagiu2}, modulo zero-order integrability corrections \cite{BCM,comi1,HO1,HO2}. This approach strictly requires homogeneity of the various estimates eventually leading to the final iteration scheme, ingredients unavoidably missing for strongly nonuniformly elliptic functionals as our log-multiphase integrals. Actually, nonhomogeneity is a distinctive feature of genuine nonuniformly elliptic problems, thus classical perturbative methods are doomed to fail. The right perspective to adopt when dealing with the regularity for nonuniformly elliptic functionals is the one introduced by Marcellini in his foundational works on so-called $(p,q)$-nonuniformly elliptic variational integrals \cite{ma2,ma4,ma1,ma5}, motivated by delicate issues in nonlinear elasticity such as the phenomenon of cavitation \cite{ma2,ma5}. These are functionals of the type
$$
W^{1,p}_{\loc}(\Omega)\ni w\mapsto \int_{\Omega}F(Dw)\, {\rm d}x,
$$
which sufficiently smooth integrand $F(\cdot)$ satisfying $(p,q)$-nonuniform ellipticity conditions
\begin{flalign*}
\left\{
\begin{array}{c}
\displaystyle
\snr{z}^{p}\lesssim F(z)\lesssim 1+\snr{z}^{q},\qquad \quad 1<p<q<\infty\\ [10pt]\displaystyle
\snr{z}^{p-2}\mathds{I}_{n\times n}\lesssim \partial^{2}F(z)\lesssim \snr{z}^{q-2}\mathds{I}_{n\times n},
\end{array}
\right.
\end{flalign*}
terminology justified by the fact that the only possible control on the pointwise ellipticity ratio is given by
$$
\mathcal{R}_{\partial F(\cdot)}(z)\lesssim 1+\snr{z}^{q-p}.
$$
In a nutshell, Marcellini's stategy consists in relating the growth of the ellipticity ratio to the amount of regularity allowed for minima: slowing down the rate of blow up of the ellipticity ratio by choosing $p$ and $q$ sufficiently close to each other turns out to be necessary and sufficient condition for regularity \cite{ma4,ma1} already for autonomous problems, so a fortiori the same holds for the models in \eqref{functional}, \eqref{multif}, \eqref{mp}, \eqref{dp}. After Marcellini's breakthrough, the regularity theory for $(p,q)$-nonuniformly elliptic problems became object of intensive investigation, see \cite{ABF,accfm,BM,BS1,bs2,bifu1,bbbb1,bbbb,bdms,cm1,cisc,dm,dp,DiMa,EMM1,EMM2,FM,ELM,gkpq,hs,ik,koch1,ko1} for an (incomplete) list of recent advances in the field, and \cite{masu2} for a very good survey. Owing to the lack of a reasonable perturbative theory, the classical approach to the regularity for nonautonomous, $(p,q)$-nonuniformly elliptic problems relies on Moser's iteration \cite{dm,DiMa,EMM1,EMM2} that unavoidably requires differentiable coefficients, thus escaping the realm of Schauder theory, in which the mere H\"older continuity of space-depending coefficients suffices to guarantee gradient H\"olderianity. In turn, nonuniformly elliptic Schauder estimates were very recently obtained by De Filippis \& Mingione \cite{DFM,DFM23} via hybrid perturbation methods based on nonlinear potential theorethic techniques and embracing both quantitative superlinear, and nearly linear growth cases. In particular, in \cite{DFM23} the authors obtain sharp regularity results for double phase problems at nearly linear growth, i.e.:
\begin{flalign}\label{nll}
W^{1,1}_{\loc}(\Omega)\ni w\mapsto \int_{\Omega}\snr{Dw}\log(1+\snr{Dw})+a(x)\snr{Dw}^{q}\, {\rm d}x,
\end{flalign}
that is \eqref{functional} with $b(\cdot)\equiv 0$. We take the techniques in \cite{DFM23} as a general blueprint and move from that to treat in an optimal fashion the much more involved case of multiple phase transitions. In this respect, we finally bring to the reader's attention two relevant features of our work.
\begin{itemize}
\item Our results are new already when coefficients are Sobolev differentiable. In this case Lipschitz regularity for minima of \eqref{dp} and \eqref{nll} can still be obtained via Moser's iteration that on the other hand fails to detect the subtle phase interactions in \eqref{q_s}, \eqref{qi}. In fact, an application of the (nonhybrid) techniques in \cite{dm,DiMa,EMM1,EMM2} would turn the bounds in \eqref{q_s}, \eqref{qi} into
$$
\max\{q,s\}<1+\frac{\min\{\alpha,\beta\}}{n},
$$
that is a trivialization of the problem.
\item Our approach allows bypassing the "separation of phases" trick usually employed when handling multiphase problems \cite{bb90,deoh}, thus drastically simplifying the whole procedure.
\end{itemize}
We conclude this section by stressing that the body of techniques developed in \cite{DFM,DFM23} will give access to several results previously available only for pointwise uniformly elliptic problems, such as interpolative Schauder theory, cf. \cite{bb90,BCM,comi2,HO3}, and the forthcoming \cite{ddp}.
\subsubsection*{Organization of the paper} The paper is organized as follows: in Section \ref{pre} we describe our notation and record several auxiliary results that will be useful in the proof of our main theorem. Section \ref{frozen} is devoted to the derivation of Lipschitz bounds for suitable frozen problem that will be fundamental in Section \ref{main}, where Theorem \ref{theorem4.1} and Corollary \ref{maincor} are proven.
\section{Preliminaries}\label{pre}
\subsection{Notation}
\noindent
In this section we specify the notation adopted. We denote by $\Omega$ an open, bounded subset of $\mathds{R}^n$, $n \geqslant 2$, and we set 
$$ B_r \equiv B_r(x_0) \coloneqq \{ x \in \mathds{R}^n : |x-x_0| < r\},$$
where, unless differently specified, all the balls considered will have the same center.  With $c$ we denote a constant not necessarily the same in any two occurrences, the relevant dependence being emphasized. We 
denote the inner hypercube of $B_r$ by $$Q_{\textnormal{inn}}\equiv Q_{\textnormal{inn}}(B_r).$$ It is the largest hypercube, contained in $B_r$, with sides parallel to the coordinate axes and concentric to $B_r$. The sidelength of $Q_{\textnormal{inn}}(B_r)$ is $2r/\sqrt{n}$. We adopt the usual definition of local minimizer.
\begin{definition}
A function $u \in W^{1,1}_{\text{loc}}(\Omega)$ is a local minimizer of $\mathcal{L}$ if and only if $x \mapsto H(x,Du(x)) \in L^1_{\text{loc}}(\Omega)$ and
$$\int_{\text{supp} \ \varphi} H(x, Du(x)) \ {\rm d}x \leqslant \int_{\text{supp} \ \varphi} H(x, Du(x) + D\varphi (x)) \ {\rm d}x,$$
for any $\varphi \in  W^{1,1}(\Omega)$ with $\text{supp} \ \varphi \Subset \Omega$.
\end{definition}
Let $\omega \in [0,1]$, for $z \in \R^k$, $k \geqslant 1$, we define the quantity
$$\vl(z) \coloneqq \sqrt{\omega^2 + |z|^2}.$$
Given a function $f: U \to \R$ and a generic set $U$, we denote
$$ (f-k)_+ \coloneqq \max \{ f-k,0\} \quad \text{and} \quad (f-k)_{-} \coloneqq \max \{ k -f,0\},$$
with $ k \in \R$. Moreover, we use the notation
$$ ||f||_{C^0,\gamma} \coloneqq ||f||_{L^{\infty}} + [f]_{0,\gamma,U}, \quad [f]_{0,\gamma,U} \coloneqq \sup_{x,y \in U, x \neq y}\frac{|f(x) - f(y)|}{|x-y|^\gamma}.$$
Finally, if $U \subset \R^n$ is a measurable subset with bounded positive measure and $f$ is a measurable map, we set
$$(f)_U \equiv \fint_{U} f(x) \, {\rm d}x \coloneqq \frac{1}{|U|} \int_U f(x) \, {\rm d}x.$$ 
%%%%%%%%%%%%%%%%%
\subsection{Other results} \label{results}
\noindent
We start recalling some facts concerning functions belonging to fractional Sobolev spaces, more details can be found in \cite{DPV}. Let us define the difference operator
$$ \tau_{s}f(x) \coloneqq f(x+h) - f(x),$$
where $f:\mathds{R}^n \rightarrow \mathds{R}$, $h \in \mathds{R}$. For $\gamma \in (0,1)$, $p \in [1,\infty)$, we define the fractional Sobolev space $W^{\gamma, p}(\Omega, \R)$ as the set of maps $f: \Omega \to \R$ such that the following Gagliardo type norm is finite: 
\begin{align*}
||f||_{W^{\gamma,p}(\Omega)} & \coloneqq ||f||_{L^p(\Omega)} + \left ( \int_\Omega \int_\Omega \frac{|f(x)-f(y)|^p}{|x-y|^{n+\gamma p}} \, {\rm d}x {\rm d}y\right )^{\frac{1}{p}} \\ & \coloneqq ||f||_{L^p(\Omega)} + [f]_{\gamma, p, \Omega}.
\end{align*}
Then, the following lemma holds true, see \cite{DFM23}.
\begin{lem} \label{2.3}
    Let $B_\rho \Subset B_r \subset \R^n $ be concentric balls with $r \leqslant 1$, $f \in L^p(B_r, \R)$, $p\geqslant 1$ and assume that, for $\alpha_0 \in (0,1]$, $S \geqslant 1$, there holds
\begin{equation*}
        ||\tau_h f||_{L^p(B_\rho)} \leqslant S|h|^{\alpha_0} \text{ for every } h \in \R, \text{ with } 0<h \leqslant \frac{r-\rho}{K}, K \geqslant 1.
\end{equation*}
    Then, for every $\gamma < \alpha_0$, it holds
\begin{equation*}
        ||f||_{W^{\gamma, p}(B_\rho)} \leqslant \frac{c}{(\alpha_0 - \gamma)^{1/p}}\left ( \frac{r-\rho}{K}\right )^{\alpha_0 - \gamma} S + c \left ( \frac{K}{r-\rho}\right )^{\frac{n}{p}+\gamma} ||f||_{L^p(B_r)},
\end{equation*}
where $c \equiv c(n,p)$.
\end{lem}
Let us also recall the fractional Sobolev embedding
\begin{equation*}
    ||f||_{L^{\frac{np}{n-p\gamma}}(B_r)} \leqslant c ||f||_{W^{\gamma, p}(B_r)},
\end{equation*}
that holds provided $p \geqslant 1$, $\gamma \in (0,1)$, and $p\gamma <n$, where $c\equiv c(n,p,\gamma)$. 
Finally, we state the next lemma, for the proof we refer to \cite[Chapter 6, Section 3]{giusti}.
\begin{lem} \label{lemma_giusti}
Let $h:[\rho, R_0] \rightarrow \mathds{R}$ be a non-negative bounded function and $0<\theta<1$,  $A, B, \gamma \geqslant 0$ be numbers. Assume that
$$h(r) \leqslant \frac{A}{(d-r)^{\gamma}} + B + \theta h(d),$$
holds for every $\rho \leqslant r < d \leqslant R_0.$ Then
$$h(\rho) \leqslant c \left ( \frac{A}{(R_0 - \rho)^{\gamma}} + B \right )$$
where $c\equiv c(\theta, \gamma) > 0.$
\end{lem}
Now, let us move on considering the auxiliary vector fields $V_{\omega, p}: \R^n \to \R^n$ defined by
\begin{equation*}
    V_{\omega, p}(z) \coloneqq (\omega^2 + |z|^2)^{\frac{p-2}{4}}z, \quad p \in [1,\infty) \text{ and } \omega \in [0,1].
\end{equation*}
For every $z_1, z_2 \in \R^n$ we have
\begin{equation} \label{v}
    |V_{\omega,p}(z_{1})-V_{\omega,p}(z_{2})|\approx ( \omega^2 +|{z_{1}|}^{2}+|z_{2}|^{2})^{(p-2)/4}|z_{1}-z_{2}|,
\end{equation}
where the equivalence holds up to constants depending only on $n$ and $p$, see \cite[Lemma 2.1]{Hamburger}. Moreover, an important property related to this field is the following
\begin{equation} \label{campo_v_1}
     (\omega^{2}+|{z_{1}}|^{2}+|{z_{2}}|^{2})^{\frac{\gamma}{2}} \approx  \int_{0}^{1}(\omega^{2}+|{z_{1}+\tau(z_{2}-z_{1})|}^{2})^{\frac{\gamma}{2}}\, d\tau,
\end{equation}
that holds for every $\gamma > -1$. If $\gamma = -1$, we get
\begin{equation} \label{campo_v_2}
(\omega^{2} + |{z_{1}}|^{2}+|{z_{2}}|^{2})^{-\frac{1}{2}} \lesssim   \int_{0}^{1}(\omega^{2}+|{z_{1}}| +\tau (z_{2}-z_{1})|^{2})^{-\frac{1}{2}} \, d \tau,
\end{equation}
see \cite{Hamburger} and \cite[Section 2]{giamod}. For $\sigma \in (0,1)$, another quantity that will play a crucial role throughout the paper is
\begin{align*}
    \mathcal{V}^{2}_{\omega, \sigma}(x,z_1,z_2) \coloneqq |V_{1,1}(z_1) - V_{1,1}(z_2)|^2 & + a_{\sigma}(x)|V_{\omega,q}(z_1) - V_{\omega,q}(z_2)|^2 \\ & + b_{\sigma}(x)|V_{\omega,s}(z_1) - V_{\omega,s}(z_2)|^2,
\end{align*}
where $a_{\sigma}(x)\coloneqq a(x) +\sigma$ and $ b_{\sigma}(x)\coloneqq b(x) +\sigma$. Moreover, for any ball $B_r \subset \Omega$ we define
\begin{align*} 
        \mathcal{V}^{2}_{\omega, \sigma, \iota}(z_1,z_2, B_r) \coloneqq |V_{1,1}(z_1) - V_{1,1}(z_2)|^2 & + a_{\sigma}^{\iota}(B_r)|V_{\omega,q}(z_1) - V_{\omega,q}(z_2)|^2 \\ & + b_{\sigma}^{\iota}(B_r)|V_{\omega,s}(z_1) - V_{\omega,s}(z_2)|^2,
\end{align*}
where $a_{\sigma}^{\iota}(B_r) \coloneqq \inf_{x \in B_r} a(x) + \sigma$ and $ b_{\sigma}^{\iota}(B_r) \coloneqq \inf_{x \in B_r} b(x) + \sigma$.
\\ \\ 
Now, let us recall some facts concerning Lorentz spaces and nonlinear potentials. Let $t, \sigma >0$, $ l, \theta \geqslant 0$ be parameters, and let $f \in L^1(B_r(x_0))$ being such that $|f|^l \in L^1(B_r(x_0))$, with $B_r(x_0) \subset \R^n$, we  consider the following non linear potential of Havin-Mazya-Wolff type \cite{HM}
\begin{equation*}
P^{l,\theta}_{t,\sigma}(f,x_0,r) \coloneqq \int_0^r \rho^\sigma \left ( \fint_{B_{\rho}(x_0)} |f|^l \, {\rm d}x \right )^{\frac{\theta}{t}} \, \frac{d \rho}{\rho}.
\end{equation*}
We record the embedding lemma for non linear potential, for the proof see \cite[Section 2.3]{DFM}.
\begin{lem} \label{lem2.4}
    Let $n \geqslant 2$, $t, \sigma, \theta >0$ be numbers such that
    $$ \frac{\theta n}{t\sigma}>1.$$
    Let $B_\rho \Subset B_{\rho + r} \subset \R^n$ be two concentric balls with $\rho, r \in (0,1]$, and let $f \in L^{1}(B_{\rho + r})$ be a function such that $|f|^l \in L^{1}(B_{\rho + r})$, where $l >0$. Then
    $$ ||P^{l,\theta}_{t,\sigma}(f,\cdot, r)||_{L^{\infty}(B_\rho)} \leqslant c ||f||^{\frac{l\theta}{t}}_{L^{\gamma}(B_{\rho + r})}$$
    holds for every $\gamma > \frac{nl\theta}{t\sigma} >0,$
    with $c \equiv c(n,t,\sigma,l,\theta,\gamma)$.
\end{lem}
Next lemma was proven in \cite[Lemma 4.2]{DFM}.
\begin{lem}\label{revlem}
Let $B_{r_0}(x_{0})\subset \mathbb{R}^{n}$ be a ball and consider functions $f_j$, 
$|f_j|^{l_j} \in L^1(B_{2r_0}(x_{0}))$, for $j = 1, 2,3, 4 $. Let $\chi >1$, $\sigma_j, l_j, \theta_j>0$ and $c_*,L_0 >0$,  $k_0, L_j\geqslant 0$ be constants. Assume that $v \in L^2(B_{r}(x_0))$ is such that for all $k \geqslant k_{0}$, and for every ball $B_{\rho}(x_{0})\subset B_{r_0}(x_{0})$, the inequality
\begin{align*}
    \left ( \fint_{B_{\rho/2}(x_0)} (v-k)_+^{2\chi} \, {\rm d}x\right )^{\frac{1}{\chi}} & \leqslant c_* L_0^2 \fint_{B_{\rho}(x_0)}(v-k)^2_+ \, {\rm d}x \\ & \quad + c_* \sum_{j=1}^4 L^2_j \rho^{2\sigma_j} \left ( \fint_{B_\rho (x_0)} |f_j|^{l_j} \, {\rm d}x \right )^{\theta_j}
\end{align*}
holds. If $x_0$ is a Lebesgue point of $v$, then
\begin{align*}
    v(x_0) & \leqslant k_0 + cL_0^{\frac{\chi}{\chi -1}}\left ( \fint_{B_{r_0}(x_0)} (v-k_0)^2_+ \, {\rm d}x \right )^{\frac{1}{2}} \\ & \quad + cL^{\frac{1}{\chi -1}} \sum_{j=1}^4 L_j P^{l_j, \theta_j}_{2,\sigma_j} (f_j, x_0, 2r_0)
\end{align*}
holds with $c\equiv c(n,\chi,\sigma_j,\theta_j,c_{*})$.  
\end{lem}
\subsubsection{Absence of Lavrentiev Gap}
\noindent 
In this last part of preliminaries we give an important approximation result for the functional $\mathcal{L}$. We observe that the functional $\vL$ does not present Lavrentiev gap, namely: every function $w \in W^{1,1}_{\text{loc}}(\Omega)$ can be approximated, also in energy, by a sequence $w_\varepsilon \in W^{1,\infty}$. This is stated in the following lemma.
\begin{lem} \label{lemma_lavrentiev}
Let $w \in W^{1,1}_{loc}(\Omega)$ be a function such that $H(\cdot, Dw) \in L^{1}_{loc}(\Omega).$ For every $B_r \Subset \Omega$, $r \in (0,1]$, there exists a sequence $$\{w_\varepsilon \}_{\varepsilon} \subset W^{1,\infty}(B_r)$$ such that $w_\varepsilon \to w$ in $W^{1,1}(B_r)$ and $$\mathcal{L}(w_\varepsilon, B_r) \rightarrow \mathcal{L}(w,B_r).$$
\end{lem}
\noindent
For more details about Lavrentiev gap we refer to \cite{ABF,Belloni-Buttazzo,ELM}. We remark that the proof closely follows that of \cite[Lemma 13]{ELM} and \cite[Section 5]{DFM23}. The substantial difference lies in the term $|Dw|\log(1+|Dw|)$. For completeness we give an outline of the proof.
\begin{proof}
Let $\varepsilon_0 \in (0,1]$ be such  that $B_{r + 2\varepsilon_0} \Subset \Omega$. Then $w \in  W^{1,1}(B_{r + 2\varepsilon_0}).$ For every $\varepsilon \in (0,\varepsilon_0)$ we denote
\begin{equation*}
    {w}_\varepsilon(x) \coloneqq  \int_{B_\varepsilon(x)} w(y)\varepsilon^{-n}\phi\left (\frac{x-y}{\varepsilon}\right) \, dy \coloneqq \int_{B_\varepsilon(x)} w(y)\phi_\varepsilon(x-y) \, dy, \quad \text{for all } x \in B_r,
\end{equation*}
where $\phi \in C^{\infty}_c(B_1(0))$, $\phi \geq 0$ and $\int_{B_1(0)} \phi(x) \, {\rm d}x =1$. It follows that ${w}_\varepsilon \in C^{\infty}(B_r)$ and
\begin{equation} \label{conv}
 {w}_\varepsilon \to w \quad \text{in } W^{1,1}(B_r).
\end{equation}
Moreover, for all $x \in B_r$ there exists a constant $c > 0$ such that
\begin{align} \label{above}
    |D {w}_\varepsilon(x)| \leq \frac{c}{\varepsilon^n}||Dw||_{L^{1}(B_{r+2\varepsilon_0})}.
\end{align}
Now, for all $x \in B_r$, we define 
$$ a_{\varepsilon}(x) \coloneqq \inf \{ a(y) : |x-y| \leq \varepsilon \},$$
$$ b_{\varepsilon}(x) \coloneqq \inf \{ b(y) : |x-y| \leq \varepsilon \},$$
and
$$H_\varepsilon(x,z) \coloneqq |z|\log(1+|z|) + a_\varepsilon(x)|z|^q + b_\varepsilon(x)|z|^s.$$
By definition, we have
\begin{align} \label{H1}
    H_\varepsilon(x,z) \leq H(y,z), \quad \text{for all } x \in B_r, \text{ for all }  y \in B_\varepsilon(x).
\end{align}
Moreover, if $|z| < \frac{c}{\varepsilon^n}$ for some $c>0$, it holds
\begin{align*}
    H_\varepsilon (x,z) & = H(x,z) + (a_\varepsilon(x)-a(x))|z|^q + (b_\varepsilon(x)-b(x))|z|^s \\ &
    \geq H(x,z) - c[a]_{0,\alpha, B_r} \varepsilon^\alpha|z|^{q-1}|z|  - c[b]_{0,\beta, B_r} \varepsilon^\beta|z|^{s-1}|z| \\ & \geq H(x,z) - c\varepsilon^{\alpha -n(q-1)}|z| -  c\varepsilon^{\beta -n (s-1)}|z|  \\ & \geq H(x,z) - c|z|\, ,
\end{align*}
where for the last inequality we have used \eqref{q_s}.
Then, for all $x \in B_r$, for all $|z|< \frac{c}{\varepsilon^n}$, it holds
\begin{equation} \label{H2}
    H(x,z) \leq (1+c)(|z| + H_\varepsilon(x,z)),
\end{equation}
with $c \equiv c(||Dw||_{L^1(B_{r + 2\varepsilon_0})},[a]_{0,\alpha,B_r}, [b]_{0,\beta,B_r})$. Now, we use the fact that $t \mapsto H_\varepsilon(x,t)$ is convex, Jensen inequality and \eqref{H1}, to obtain
\begin{align*}
H_\varepsilon(x,D {w}_\varepsilon(x)) & \leq \int_{B_\varepsilon(x)} H_\varepsilon(x,Dw(y))\phi_\varepsilon(x-y) \, dy \\ & \leq \int_{B_\varepsilon(x)} H(y,Dw(y))\phi_\varepsilon(x-y) \, dy  \\ & \eqcolon(H(\cdot,Dw(\cdot)))_\varepsilon (x).
\end{align*}
Bearing in mind \eqref{above} and \eqref{H2}, we have
\begin{equation*}
    H(x,D {w}_\varepsilon(x)) \leq c[|D {w}_\varepsilon(x)| +(H(\cdot,Dw(\cdot)))_\varepsilon (x)].
\end{equation*}
By \eqref{conv} and
$$ (H(\cdot,Dw(\cdot)))_\varepsilon (x) \to H(x,Dw(x)) \quad \text{strongly in } L^{1}(B_r) $$
we can use well-known variant of Lebesgue’s dominated convergence theorem to obtain that 
$$ H(x,D {w}_\varepsilon(x)) \to H(x,Dw(x))   \quad \text{strongly in } L^{1}(B_r).$$
This ends the proof.
\end{proof}
\section{$L^{\infty}$-estimate for the frozen problem} \label{frozen}
\noindent
 In this section we construct densities $H\os$ representing a regularized version of $H$, in this way we transform the starting problem into an elliptical one. At this point, we consider the frozen integrand $H\ios$ and, using the fact that the coefficients $a\ios, b\ios$ can be differenciate with continuity, we obtain an estimate for $||Dv\ios||_{L^{\infty}}$, where $v\ios$ is the minimizer of $H\ios$.

Let us start considering a ball $B_r \Subset \Omega$ and two parameters $\omega \in (0,1], \sigma \in (0,1)$. For $(x,z) \in B_r \times \R^n$ we define
$$ H_{\omega,\sigma}(x,z) \coloneqq \vl(z)\log(1+\vl(z)) + a_\sigma(x)(\vl(z))^q + b_\sigma(x)(\vl(z))^s,$$
where $a_\sigma$ and $b_\sigma$ have been introduced in Section \ref{results}. We observe that $z \mapsto \partial_{zz}H(x,z) \in C(\R^n)$. It is easy to check that, for every $x \in B_r$ and for every $z \in \R^n$, there exists a constant $c\equiv c(n,q,s,||a||_{C^{0,\alpha}}, ||b||_{C^{0,\beta}})$ such that
\begin{equation} \label{h_1}
\sigma \left [(\vl(z))^q + (\vl(z))^s \right ] \leqslant H\os(x,z) \leqslant c \left [ 1+(\vl(z))^q + (\vl(z))^s \right ],  
\end{equation}
\begin{equation} \label{h_2}
    \langle \partial_{zz} H\os(x,z)\xi,\xi \rangle \geqslant c\sigma \left [\ (\vl(z))^{q-2} + (\vl(z))^{s-2} \right]|\xi|^2,
\end{equation}
\begin{equation} \label{h_3}
    |\partial_z H\os(x,z)| \vl(z) + |\partial_{zz}H\os(x,z)| (\vl(z))^2 \leqslant c\left [ (\vl(z))^q + (\vl(z))^s \right],
\end{equation}
\begin{equation} \label{h_4}
     |\partial_z H\os(x,z) - \partial_z H\os(y,z)| \leqslant c \left [ |x-y|^\alpha (\vl(z))^{q-1} + |x-y|^\beta (\vl(z))^{s-1}\right],
\end{equation}
where, for the last property, we have used the fact that $a \in C^{0,\alpha}(\Omega)$ and $b \in C^{0,\beta}(\Omega)$. Now, let us determine two functions $\lambda\os, \Lambda\os: \Omega \times [0, \infty) \to [0,\infty)$ which represent respectively a bound for the lowest and highest eigenvalue of $\partial_{zz}H\os$. By direct computation we have
\begin{align*}
    \partial_z H\os(x,z) = \frac{z}{\vl(z)}\left [ \log(1+\vl(z)) + \frac{\vl(z)}{1+\vl(z)}\right] &+ qa_\sigma (x)(\vl(z))^{q-2}z \\ &+ sb_\sigma (x)(\vl(z))^{s-2}z,
\end{align*}
\begin{align*}
    \partial_{zz}H\os(x,z)= & \left [ \mathds{I}_{n\times n} + \frac{z\otimes z}{(\vl(z))^2} \right] \frac{\log(1+\vl(z))}{\vl(z)} \\ & \quad + \frac{\mathds{I}_{n\times n}}{1+\vl(z)} + \frac{z\otimes z}{(\vl(z))^2(1+\vl(z))^2} \\ & \quad + qa_\sigma(x)(\vl(z))^{q-2}\left [\mathds{I}_{n\times n} + (q-2)\frac{z\otimes z}{(\vl(z))^2}\right] \\ & \quad + sb_\sigma(x)(\vl(z))^{s-2}\left [\mathds{I}_{n\times n} + (s-2)\frac{z\otimes z}{(\vl(z))^2}\right].
\end{align*}
Then,
\begin{align*}
    \langle \partial_{zz}H\os(x,z) \xi, \xi \rangle & = \left [ |\xi|^2 - \frac{|z\cdot \xi|^2}{(\vl(z))^2}\right ] \frac{\log(1+(\vl(z)))}{\vl(z)}  \\ & \quad + \frac{|\xi|^2}{1+\vl(z)} + \frac{|z\cdot \xi|^2}{(\vl(z))^2(1+\vl(z))^2}  \\ &  \quad + qa_\sigma(x)(\vl(z))^{q-2} \left [ |\xi|^2 + (q-2)\frac{|z\cdot \xi|^2}{(\vl(z))^2} \right ] \\ &  \quad + sb_\sigma(x)(\vl(z))^{s-2} \left [ |\xi|^2 + (s-2)\frac{|z\cdot \xi|^2}{(\vl(z))^2} \right ],
\end{align*}
for all $\xi \in \R^n$, and
\begin{align*}
    |\partial_{zz}H\os(x,z)| \leqslant \frac{2n^2 \log(1+\vl(z))}{\vl(z)} &+ \frac{2n^2}{1+\vl(z)} \\ & + qn^2a_\sigma (x)\max\{1,q-1\}(\vl(z))^{q-2} \\ & + sn^2b_\sigma (x)\max\{1,s-1\}(\vl(z))^{s-2}.
\end{align*}
Hence, taking
\begin{align*}
     \lambda\os(x,|z|) \coloneqq \frac{1}{1 + \ell_{\omega}(z)} & + q\min\{1,q-1\} a_\sigma(x)(\vl(z))^{q-2} \\ & + s\min\{1,s-1\} b_\sigma(x)(\vl(z))^{s-2},
\end{align*}
\begin{align*}
    \Lambda\os(x,|z|) & \coloneqq 2n^2\left (\frac{\log(1+\vl(z))}{\vl(z)} + \frac{1}{1+\vl(z)} \right ) \\ & \quad + qn^2\max\{1,q-1\} a_\sigma(x)(\vl(z))^{q-2} \\ & \quad + sn^2\max\{1,s-1\} b_\sigma(x)(\vl(z))^{s-2},
\end{align*}
we have
\begin{equation} \label{h_lambda}
      \lambda\os(x,|z|)|\xi|^2 \leqslant \langle \partial_{zz} H\os(x,z) \xi, \xi \rangle, \quad   |\partial_{zz} H\os(x,z)| \leqslant |\Lambda\os(x,|z|)|,
\end{equation}
for all $\xi \in \R^n$. In particular, it holds 
\begin{align*}
    \frac{\Lambda\os(x,|z|)}{\lambda\os(x,|z|)} \leqslant c(n,q,s)\log(1+\vl(z)), \quad \text{for all } z \in \R^n : |z|\geqslant e-1.
\end{align*}
Moreover, we observe that using \eqref{campo_v_1}, \eqref{campo_v_2} and the first inequality in \eqref{h_lambda}, we get
\begin{equation*}
     \langle \partial_z H\os(x,z_1) - \partial_z H\os(x,z_2), z_1 - z_2 \rangle \geqslant c\mathcal{V}^2_{\omega,\sigma}(x,z_1,z_2),
\end{equation*}
for all $x \in B_r$, for all $z,z_1,z_2 \in \R^n$, where $c\equiv c(n,q,s,||a||_{C^{0,\alpha}}, ||b||_{C^{0,\beta}})$. \\ \indent
Now, let us consider the frozen integral
$$ H\ios(z) \coloneqq \vl(z)\log(1+\vl(z)) + a\is(B_r)(\vl(z))^q +  b\is(B_r)(\vl(z))^s, $$
where $a\is$ and $b\is$ have been defined in the preliminaries. For $H\ios$ are in force the same properties of $H\os$, namely $z \mapsto \partial_{zz}H\ios \in C(\R^n)$ and
\begin{equation} \label{H_froz_1}
\sigma \left [(\vl(z))^q + (\vl(z))^s \right ] \leqslant H\ios(z) \leqslant c \left [ 1+(\vl(z))^q + (\vl(z))^s \right ], 
\end{equation}
\begin{equation} \label{H_froz_2}
    \langle \partial_{zz} H\ios(z)\xi,\xi \rangle \geqslant c\sigma \left [\ (\vl(z))^{q-2} + (\vl(z))^{s-2} \right]|\xi|^2,
\end{equation}
\begin{equation} \label{H_froz_3}
    |\partial_z H\ios(z)| \vl(z) + |\partial_{zz}H\ios(z)| (\vl(z))^2 \leqslant c\left [ (\vl(z))^q + (\vl(z))^s \right],
\end{equation}
\begin{equation} \label{H_froz_4}
     \langle \partial_z H\ios(z_1) - \partial_z H\ios(z_2), z_1 - z_2\rangle \geqslant c\mathcal{V}^2_{\omega,\sigma,\iota}(z_1,z_2,B_r),
\end{equation}
for all $x \in B_r$, for all $z,z_1,z_2 \in \R^n$, where $c \equiv c(n,q,s,||a||_{C^{0,\alpha}}, ||b||_{C^{0,\beta}})$. Similarly, denoting by $\lambda\ios$ and $\Lambda\ios$ respectively the bound for the lowest and the highest eigenvalue of $\partial_{zz}H\ios$, for all $\xi \in \R^n$, we have
\begin{equation}  \label{H_froz_5}
   \lambda\ios(|z|)|\xi|^2 \leqslant \langle \partial_{zz} H\ios(z) \xi, \xi \rangle, \quad 
     |\partial_{zz} H\ios(z)| \leqslant \Lambda\ios(|z|), 
\end{equation}
\begin{equation} \label{frozen_autovalori}
    \frac{\Lambda\ios(|z|)}{\lambda\ios(|z|)} \leqslant c(n,q,s)\log(1+\vl(z)), \quad \text{for all } z \in \R^n : |z|\geqslant e-1.
\end{equation}
Now, another quantity that will play a crucial role in the proof of the Lipschitz estimate is
$$ E\os(x,t)\coloneqq \int_{0}^t \lambda\os(x,\Tilde{t})\Tilde{t} \, {\rm d}\Tilde{t}, $$
where $x \in \Omega$ and $t \geqslant 0$, with the frozen counterpart defined as
$$ E\ios(t)\coloneqq \int_{0}^t \lambda\ios(\Tilde{t})\Tilde{t} \, {\rm d}\Tilde{t}. $$
By direct computations we get
\begin{align*}
    E\os(x,t) \coloneqq \vl(t) - \log(1+\vl(t)) & + \min\{1,q-1\}a_\sigma(x)(\vl(t))^q \\ & + \min\{1,s-1\}b_\sigma(x)(\vl(t))^s - \Tilde{e}\os(x),
\end{align*}
\begin{align*}
    E\ios(t) \coloneqq \vl(t) - \log(1+\vl(t)) & + \min\{1,q-1\}a\is(B_r)(\vl(t))^q \\ & + \min\{1,s-1\}b\is(B_r)(\vl(t))^s - \Tilde{e}\ios,
\end{align*}
where
\begin{align*}
    \Tilde{e}\os(x) \coloneqq \omega - \log(1+\omega) + \min\{1,q-1\}a_\sigma(x)\omega^q + \min\{1,s-1\}b_\sigma(x)\omega^s
\end{align*}
and
\begin{align*}
    \Tilde{e}\ios \coloneqq \omega - \log(1+\omega) + \min\{1,q-1\}a\is(B_r)\omega^q + \min\{1,s-1\}b\is(B_r)\omega^s.
\end{align*}
We can also derive the following inequalities
\begin{align} \label{in1}
  &  |E\ios(t) - E\ios(\Tilde{t})| \notag \\ & \qquad  \leqslant \left [ 1 + a^{\iota}_\sigma(B_r)(t^2 +\Tilde{t}^2 + \omega^2)^{\frac{q-1}{2}} + b^{\iota}_\sigma(B_r)(t^2 +\Tilde{t}^2 + \omega^2)^{\frac{s-1}{2}}\right]|t-\Tilde{t}|,
\end{align}
\begin{align} \label{3.10}
    |E\os(x,t) - E\ios(t)|& \leqslant c(q)|a(x) - a^\iota(B_r)|\left [(\vl(t))^q -\omega^q\right] \notag \\ & \quad + c(s)|b(x) - b^\iota(B_r)|\left [(\vl(t))^s -\omega^s\right],
\end{align}
that hold for all $x \in \Omega$, for all $t,\Tilde{t} \geqslant 0$. Moreover, we can see that there exists a number $c_0 \in [1, \infty)$ such that, for all $t \geqslant 1$,
\begin{align} \label{crescita}
   \vl(t) \leqslant c_0 E\os(x,t) \quad \text{and} \quad  \vl(t) \leqslant c_0 E\ios(t).
\end{align}
Finally, let us define two more quantities that will be used later
\begin{equation} \label{tmq} 
\begin{split}
    \Tilde{E}\os(x,t) & \coloneqq E\os(x,t) + \log(1+\vl(t)) + \Tilde{e}\os(x),\\
    \Tilde{E}\ios(t) & \coloneqq E\ios(t) + \log(1+\vl(t)) + \Tilde{e}\ios.
    \end{split}
\end{equation}
In particular, we observe that,
\begin{equation} \label{tilde_e_1}
\Tilde{E}\os(x,|z|) \leqslant E\os(x,|z|) + \log(1+\vl(z)) + \Tilde{e}\os(x),
    \end{equation}
\begin{equation} \label{tilde_e}
\Tilde{E}\ios(|z|) \leqslant E\ios(|z|) + \log(1+\vl(z)) + \Tilde{e}\ios.
    \end{equation}
We are, now, ready to prove the following theorem.
% Using the fact that the frozen problem is elliptic and has the coefficients $a\ios, b\ios$ continuously differentiable, we obtain an $L^{\infty}$-estimate on the minimizer of the functional $$\int_{B_r} H\ios (Dw) \, {\rm d}x.$$ 
From now on, we take
\begin{equation}\label{mu}
    \mu \coloneqq \max \{q,s\}.
\end{equation}
% where $B_r \Subset \Omega$, $r \in (0,1]$, and $u_0 \in$ is a function in $W^{1,\infty}(B_r)$.
\begin{thm} \label{teorema3.1}
    Let $B_r \Subset \Omega$, $r \in (0,1]$, and let $u_0 \in W^{1,\infty}(B_r)$. Then there exists an unique solution $v\ios \in u_0 + W^{1,\mu}_0(B_r)$ of the Dirichlet problem
    \begin{equation} \label{dirichlet_1}
     u_0 + W^{1,\mu}_0(B_r)  \ni w \mapsto \min\int_{B_r} H\ios(Dw) \, {\rm d}x.
    \end{equation}
    Moreover, for any $\delta \in (0,1)$, there exists a constant $c \equiv c(n,q,s,||a||_{C^{0,\alpha}}, ||b||_{C^{0,\beta}},\delta)$ such that
    \begin{equation} \label{stima_v}
        ||Dv\ios||_{L^{\infty}(B_{3r/4})} \leqslant c\left \{\left [ H\ios(||Du_0||_{L^{\infty}(B_r)})\right]^{\delta} ||Du_0||_{L^{\infty}(B_r)} + 1 \right \}.
    \end{equation}
\end{thm}
\begin{proof}
    We start observing that by Direct methods of the Calculus of Variations there exists a solution $v\ios \in u_0 + W^{1,\mu}_0(B_r)$ of the Dirichlet problem \eqref{dirichlet_1}. The uniqueness follows by the strict convexity of $w \mapsto \int_{B_r} H\ios(Dw) \, {\rm d}x$. By minimality, the following Euler-Lagrange equation holds
    \begin{equation} \label{EL}
         \int_{B_r} \langle \partial H\ios (Dv\ios),D\varphi \rangle \, {\rm d}x =0,
    \end{equation}
    for all $\varphi \in W^{1,\mu}_0(B_r)$. By \eqref{H_froz_1}-\eqref{H_froz_3} and standard regularity theory \cite{giagiu, giusti} we have
    \begin{equation*}
        v\ios \in W^{1,\infty}_{\text{loc}}(B_r) \cap W^{2,2}_{\text{loc}}(B_r) \quad \text{and} \quad \partial_z H{\ios}(Dv\ios) \in W^{1,2}_{\text{loc}} (B_r, \R^n).
    \end{equation*}
So, we can differenciate again \eqref{EL} to get
\begin{equation} \label{EL2}
    \sum_{s=1}^n \int_{B_r} \langle \partial_{zz}H\ios(Dv\ios)D_s Dv\ios, D\varphi \rangle \, {\rm d}x =0,
\end{equation}
for all $\varphi \in W^{1,2}(B_r)$ with $\text{supp}(\varphi) \Subset B_r.$ \\
 \indent For simplicity of notation we omit the parameters $\omega, \sigma$ since the estimates obtained will be independent of them.  
% \indent Now, let $B_\rho \Subset B_r$, $\rho < r$, and let $\eta \in C^{1}_c(B_r)$ be a cut-off function such that $$\mathds{1}_{B_{3\rho/4}} \leqslant \eta \leqslant \mathds{1}_{B_{5\rho/6}} \quad \text{and} \quad |D\eta| \lesssim \left (\frac{1}{|B_\rho|} \right )^{\frac{1}{n}}.$$
% We substitute the admissible test function $\varphi \coloneqq \eta^2(\ei(|Dv|) - \k)_+ D_sv$ in \eqref{EL2} and 
\\ \indent 
Now, let $\rho >0$ be such that $B_\rho \Subset B_r$. Recalling the bound \eqref{frozen_autovalori}, we apply Lemma 4.5 in \cite{BM} with $G_{\overline{T}} = \ei$ and $f=0$ obtaining
\begin{align} \label{bm}
    \int_{B_{\rho/2}} |D(\ei(|Dv^\iota|) -k)_+|^2 \, {\rm d}x \leqslant \frac{c\log(1+\vl(M))}{\rho^2} \int_{B_\rho}(\ei(|Dv^\iota|)-k)_+^2 \, {\rm d}x,
\end{align}
for some $M \geqslant \max\{||Dv^\iota||_{L^{\infty}({B_{\rho}})},e -1 \}$, with $c\equiv c(n,q,s)$. We point out that equation (4.32) in \cite{BM} is our \eqref{EL2}. Now, label \eqref{bm} in conjunction with Sobolev embedding theorem led to
\begin{equation*}
    \left ( \fint_{B_{\rho/2}}(\ei(|Dv^\iota|) - k)_+^{2\chi} {\rm d}x \right )^{\frac{1}{\chi}} \leqslant c\log(1+\vl(M))\fint_{B_\rho} (\ei(|Dv^\iota|)-k)_+^2 \, {\rm d}x,
\end{equation*}
where $\chi \equiv \chi(n) >1$. At this point we fix parameters $3r/4 \leqslant  d_1 < d_2 \leqslant 5r/6$, we take $x_0 \in B_{d_1}$ and  $r_0 \coloneqq (d_2 - d_1)/8$.  Without loss of generality we can assume $||Dv^\iota||_{L^{\infty}(B_{3r/4})} \geqslant e-1$, otherwise \eqref{stima_v} follows trivially. Then we can choose $M = ||Dv^\iota||_{L^{\infty}({B_{d_2}})}$. We apply Lemma \ref{revlem} with $k_0 = 0$ to get
\begin{equation*}
    \ei(|Dv^\iota(x_0)|) \leqslant c[\log(1+\vl(M))]^{\frac{\chi}{\chi-1}} \left ( \fint_{B_{r_0}} \ei(|Dv^\iota|)^2 \, {\rm d}x\right )^{\frac{1}{2}}.
\end{equation*}
Since $x_0$ is arbitrary, we have
\begin{equation} \label{dopo_young}
     \ei(||Dv^\iota||_{L^{\infty}(B_{d_1})}) \leqslant \frac{c[\log(1+\vl(M))]^{\frac{\chi}{\chi-1}}}{(d_2 - d_1)^{n/2}} \left ( \int_{B_{5r/6}} \ei(|Dv^\iota|)^2 \, {\rm d}x\right )^{\frac{1}{2}}.
\end{equation}
At this stage we recall the following key-property of logarithms 
\begin{equation} \label{log}
  \log(1+t) \leqslant \frac{(1+t)^{\epsilon}}{{\epsilon}},
\end{equation}
that holds for all $\epsilon >0$ and $t \geqslant 1$, and we observe that for all $\delta \in (0,1)$ it is possible to find $\epsilon >0$ such that
$$ \epsilon\frac{\chi}{\chi -1} < \frac{\delta}{2(2+\delta)}.$$
Using the information in the last display, we apply \eqref{crescita} and \eqref{log} to \eqref{dopo_young}, and we get
\begin{align*}
    \ei(||Dv^\iota||_{L^{\infty}(B_{d_1})}) & \leqslant \frac{c[\log(1+\vl(||Dv^\iota||_{L^{\infty}(B_{d_2})})]^{\frac{\chi}{\chi-1}}}{(d_2 - d_1)^{n/2}} \left ( \int_{B_{5r/6}} \ei(|Dv^\iota|)^2 \, {\rm d}x\right )^{\frac{1}{2}} \\ & \leqslant \frac{c(\vl(||Dv^\iota||_{L^{\infty}(B_{d_2})}))^{\frac{\delta}{2(2+\delta)}}}{(d_2 - d_1)^{n/2}} \left ( \int_{B_{5r/6}} \ei(|Dv^\iota|)^2 \, {\rm d}x\right )^{\frac{1}{2}} \\ & \leqslant \frac{c(\ei(||Dv^\iota||_{L^{\infty}(B_{d_2})}))^{\frac{1+\delta}{2+\delta}}}{(d_2- d_1)^{n/2}} \left ( \int_{B_{5r/6}} (1+H^\iota(Dv^\iota)) \, {\rm d}x\right )^{\frac{1}{2}} \\ & \leqslant \frac{\ei(||Dv^\iota||_{L^{\infty}(B_{d_2})})}{2}  \\ & \quad + \frac{c}{(d_2 - d_1)^{n(1+\delta/2)}} \left ( \int_{B_r} (1+H^{\iota}(Du_0)) \, {\rm d}x \right )^{1+\frac{\delta}{2}},
\end{align*}
with $c \equiv c(n,q,s,\delta)$; for the last inequality we used the minimality of $v^\iota$ and Young's inequality with the conjugate exponents $$ \left ( \frac{2+\delta}{1+\delta}\right ) \quad \text{and} \quad  (2+\delta).$$
Now, applying Lemma \ref{lemma_giusti}, we get
$$ E^\iota(||Dv^{\iota}||_{L^{\infty}(B_{3r/4})})   \leqslant c \left ( \fint_{B_r} (1+H^\iota(Du_0)) \, {\rm d}x \right )^{1+\frac{\delta}{2}}. $$
By \eqref{tilde_e}, we have
\begin{align*}
    \Tilde{E}^{\iota}(||Dv^\iota||_{L^{\infty}(B_{3r/4})}) & \leqslant E^\iota(||Dv^\iota||_{L^{\infty}(B_{3r/4})}) + \log(1+\vl(||Dv^\iota||_{L^{\infty}(B_{3r/4})})) + \Tilde{e}^\iota
    \\ & \leqslant  c \left ( \fint_{B_r} (1+H^\iota(Du_0)) \, {\rm d}x \right )^{1+\frac{\delta}{2}} + \frac{ \Tilde{E}^{\iota}(||Dv^\iota||_{L^{\infty}(B_{3r/4})})}{2} + \Bar{c},
\end{align*}
for some positive constant $\Bar{c} \equiv \Bar{c}(n,q,s,\delta)$, which implies
\begin{align} \label{ee}
     \Tilde{E}^{\iota}(||Dv^\iota||_{L^{\infty}(B_{3r/4})}) & \leqslant c\left[H^{\iota}(||Du_0||_{L^{\infty}(B_r)})\right]^{1+\frac{\delta}{2}} + \Bar{c}   \notag \\ & \leqslant c\left[H^{\iota}(||Du_0||_{L^{\infty}(B_r)})\right]^{\frac{\delta}{2}}\log(1+\vl(||Du_0||_{L^{\infty}(B_r)})) \notag\\ & \quad \times  \vl(||Du_0||_{L^{\infty}(B_r)}) \notag \\ & \quad + c\left[H^{\iota}(||Du_0||_{L^{\infty}(B_r)})\right]^{\frac{\delta}{2}} \left [ a\is(B_r)(\vl(||Du_0||_{L^{\infty}(B_r)}))^q \right. \notag \\ & \quad  + \left . b\is(B_r)(\vl(||Du_0||_{L^{\infty}(B_r)}))^s \right] + \Bar{c}\notag \\ & \leqslant c\left \{\left [ H^{\iota}(||Du_0||_{L^{\infty}(B_r)})\right]^{\delta}\Tilde{E}^{\iota}(||Du_0||_{L^{\infty}(B_r)}) + 1 \right \},
\end{align}
 where $c\equiv c(n,q,s,||a||_{C^{0,\alpha}}, ||b||_{C^{0,\beta}},\delta)$; for the last inequality we used \eqref{log} with $\epsilon = \delta/2$. Now, we set $\Tilde{E}_0(t) \coloneqq t + a\is(B_r)t^q+b\is(B_r)t^s$ and we notice that
$$ \Tilde{E}_0(t) \leqslant \Tilde{E}^{\iota}(t) \leqslant c(q,s)(\Tilde{E}_0(t) + \Tilde{E}_0(\omega)).$$
So, inequality \eqref{ee} reads as
\begin{equation} \label{last_ee}
\begin{split}
   & \Tilde{E}_0(||Dv^\iota||_{L^{\infty}(B_{3r/4})}) \\
   & \leqslant c\left \{\left [ H^{\iota}(||Du_0||_{L^{\infty}(B_r)})\right]^{\delta} \Tilde{E}_0(||Du_0||_{L^{\infty}(B_r)}) + 1 \right \}.
    \end{split}
\end{equation}
Now, to conclude, we observe that $\Tilde{E}_0(\cdot)$ is monotone, increasing, convex and $\Tilde{E}_0(0)=0$, so, the inverse $\Tilde{E}_0^{-1}$ is increasing, concave and $\Tilde{E}_0^{-1}(0)=0$. Then, by subadditivity, we have $\Tilde{E}_0^{-1}(c_1 t) \leqslant (c_1 +1)\Tilde{E}_0^{-1}(t)$, for any constant $c_1 \geqslant 0$, see \cite[Section 4]{DFM23}. Therefore we can apply $\Tilde{E}_0^{-1}$ on both sides of \eqref{last_ee} achieving \eqref{stima_v}.

\end{proof}

\section{Main Theorem}\label{main}
\noindent
This section is devoted to the proof of the main result of the present paper, that is Theorem \ref{theorem4.1}.
Assuming the existence of a minimizer $u \in W^{1,1}$ of the functional \eqref{functional}, we show that $u$ has locally H\"older continuous gradient by showing that it can be locally approximated, in energy, by more regular maps satisfying suitable uniform Lipschitz bounds. Once proven gradient boundedness, the amount of nonuniform ellipticity of \eqref{functional} becomes immaterial and gradient H\"older continuity follows by more standard means.
\subsection*{Proof of Theorem \ref{theorem4.1}}
For the ease of exposition, we divide the proof into eight steps.
\\ \\
{\it Step 1}: In this step we find two sequences $\{Du\oe\}_\omega$ and $\{u_\varepsilon\}_\varepsilon$ such that 
$$u\oe \rightharpoonup u_\varepsilon \quad \text{weakly in  } W^{1,\mu}(B_r) \quad \text{and} \quad u_\varepsilon \rightharpoonup u \quad \text{weakly in } W^{1,1}(B_r).$$
% Let $\omega =\{\omega \} = \{ \omega_k\}_k$ and $\varepsilon =\{\varepsilon \} = \{ \varepsilon_k\}_k$ be two decreasing sequence such that 
We start applying Lemma \ref{lemma_lavrentiev} in order to find a sequence $\{ \Tilde{u}_{\varepsilon} \} \in W^{1,\infty}(B_r)$ such that
\begin{equation} \label{convergence}
  \Tilde{u}_{\varepsilon} \to u \text{ in } W^{1,1}(B_r) \quad \text{and} \quad \mathcal{L}(\Tilde{u}_\varepsilon, B_r) \to \mathcal{L}(u, B_r).
  \end{equation}
Let $ \omega \equiv \{\omega \} \equiv \{\omega_k \}_{k \in \N} \in (0,1]$ be a decreasing sequence such that $\omega \to 0$ and let $\{\sigma_\varepsilon \}$ be defined by
$$ \sigma_\varepsilon \coloneqq \left( 1+ \varepsilon^{-1} + ||D\Tilde{u}_\varepsilon||^{2q}_{L^q(B_r)} + ||D\Tilde{u}_\varepsilon||^{2s}_{L^s(B_r)}\right)^{-1},$$
then
\begin{equation} \label{sigma}
    \sigma_{\varepsilon} \int_{B_r} [( \ell_{\omega}(D\Tilde{u}_\varepsilon))^q + ( \ell_{\omega}(D\Tilde{u}_\varepsilon))^s ] \, {\rm d}x \to 0.
\end{equation}
We consider the Dirichlet problem
\begin{equation} \label{dirichlet}
     \Tilde{u}_{\varepsilon} + W^{1,\mu}_{0}(B_r) \ni w \mapsto \min \mathcal{L}_{\omega, \varepsilon} (w,B_r),
\end{equation}
where
$$ \mathcal{L}_{\omega,\varepsilon}(w,B_r) \coloneqq \int_{B_r} H_{\omega, \sigma_\varepsilon}(x, Dw) \, {\rm d}x,$$
we recall that $H_{\omega, \sigma_\varepsilon}$ has been defined at the beginning of Section \ref{frozen} with $\sigma = \sigma_\varepsilon$.
By Direct Methods of the Calculus of Variations and standard strict convexity arguments there exists $u_{\omega, \varepsilon} \in  \Tilde{u}_{\varepsilon} + W^{1,\mu}_{0}(B_r) $ unique solution of \eqref{dirichlet}. By \eqref{h_1}-\eqref{h_3} we can apply the regularity theory to get
\begin{equation} \label{holder}
    u_{\omega, \varepsilon} \in C^{1,\Tilde{\gamma}}_{\text{loc}}(B_r) \quad \text{for some } \Tilde{\gamma} \equiv \Tilde{\gamma}(n,q,s,\omega,\varepsilon) \in (0,1),
\end{equation}
see \cite{giagiu} and \cite{giagiu2}. Now, we use \eqref{sigma} and the generalized dominate convergence theorem to obtain
\begin{align*}
| \mathcal{L}_{\omega,\varepsilon}&(\Tilde{u}_\varepsilon, B_r) -  \mathcal{L}(\Tilde{u}_\varepsilon, B_r)| \\ & \leqslant \int_{B_r} |\ell_\omega(D\vu)\log(1+\ell_\omega(D\vu)) - |D\vu|\log(1+|D\vu|)| \, {\rm d}x \\ & \quad + \int_{B_r} a(x)|(\vl(D\vu))^q - |D\vu|^q| \, {\rm d}x + \sigma_\varepsilon \int_{B_r} (\vl(D\vu))^q \, {\rm d}x \\ & \quad + \int_{B_r} b(x)|(\vl(D\vu))^s - |D\vu|^s| \, {\rm d}x + \sigma_\varepsilon \int_{B_r} (\vl(D\vu))^s \, {\rm d}x   \\ & \leqslant o_\varepsilon(\omega) + o(\varepsilon),
\end{align*}
where $o_\varepsilon(\omega)$ and $o(\varepsilon)$ denote quantities such that $o_\varepsilon(\omega) \to 0$ if $\omega \to 0$, for every fixed $\varepsilon$, and $o(\varepsilon) \to 0$ if $\varepsilon \to 0$. Now,  we use the information in the last display together with the minimality and \eqref{convergence} to get
\begin{align} \label{con_l}
    \mathcal{L}_{\omega,\varepsilon}(u_{\omega,\varepsilon}, B_r) \notag & \leqslant \vL_{\omega,\varepsilon}(\vu, B_r) \\ \notag & \leqslant \vL(\vu,B_r) + |\mathcal{L}_{\omega,\varepsilon}(\vu,B_r) - \mathcal{L}(\vu, B_r)| \\  & \leqslant \vL(u,B_r) + o_\varepsilon(\omega) + o(\varepsilon). 
\end{align}
Then,  for every $\varepsilon \in (0,1)$ the sequence $\{u_{\varepsilon,\omega} \}_\omega$ is uniformly bounded in $W^{1,\mu}(B_r)$, therefore, up to not relabelled subsequences, we may suppose that there exists $u_{\varepsilon}$ such that 
\begin{equation} \label{con}
    u_{\omega,\varepsilon} \rightharpoonup u_{\varepsilon} \text{\ \ weakly in $W^{1,\vm}(B_r)$} \quad \text{and} \quad u_{\varepsilon} - \vu \in W^{1,\vm}_0(B_r).
\end{equation}
Then, for $\omega \to 0$, inequality \eqref{con_l} implies
\begin{equation} \label{5.8}
    \int_{B_r} H(x,Du_\varepsilon) \, {\rm d}x + \sigma_\varepsilon \int_{B_r} |Du_\varepsilon|^q \, {\rm d}x + \sigma_\varepsilon \int_{B_r} |Du_\varepsilon|^s \, {\rm d}x \leqslant \vL(u, B_r) + o(\varepsilon),
\end{equation}
for every $\varepsilon \in (0,1)$. Hence, the sequence $\{|Du_\varepsilon|\log(1+|Du_\varepsilon|)\}_\varepsilon$ is uniformly bounded in $L^1(B_r)$, therefore by Dunford-Pettis theorem we can conclude that there exists $\Bar{u} \in W^{1,1}(B_r)$ such that
\begin{equation*}
    u_{\varepsilon} \rightharpoonup \Bar{u} \text{\ \ weakly in $W^{1,1}(B_r)$} \quad \text{and} \quad \Bar{u} - u \in W^{1,1}_0(B_r).
\end{equation*}
So, sending $\varepsilon \to 0$ and using weak lower semicontinuity, label \eqref{5.8} turns into $\vL(\Bar{u}, B_r) \leqslant \vL(u, B_r)$. Recalling that $u$ is the minimizer, we get $\vL(\Bar{u}, B_r) = \vL(u, B_r)$. The strict convexity of the functional $\mathcal{L}$ implies $u \equiv \Bar{u}$ in $B_r$. Therefore, we can conclude that
\begin{equation} \label{con2}
    u_\varepsilon \rightharpoonup u \quad \text{weakly in } W^{1,1}(B_r).
\end{equation}
{\it Step 2:} We rescale the functions $u_{\omega,\varepsilon}$ and $H_{\omega,\varepsilon}$ on $B_1(0).$ 
\\ \\ 
Let us fix $\omega \in (0,1]$ and $\varepsilon \in (0,1)$, and let $B_{\rho}(x_0) \Subset B_r$ be any ball. We blow $u_{\omega, \varepsilon}$ and $H_{\omega,\sigma_\varepsilon}$ on $B_{\rho}(x_0)$ by defining $$ u_{\rho}(x) \coloneqq u_{\omega, \varepsilon}(x_0 + \rho x)/{\rho},$$ 
$$ H_{\rho}(x,z) \coloneqq H_{\omega, \sigma_\varepsilon}(x_0+\rho x,z).$$
By {\it Step 1}, $u_{\rho} \in W^{1,\vm}(B_1(0))$ (with $\mu$ given by \eqref{mu}) is a local minimizer of
$$ W^{1,\vm}(B_1(0)) \ni w \mapsto \int_{B_1(0)} H_{\rho}(x,Dw) \, {\rm d}x.$$
% and so, by \eqref{holder},
% \begin{equation*}
%     u_\rho \in C^{1,\Tilde{\gamma}}(B_1(0)).
% \end{equation*}
Moreover, $u_\rho$ satisfies the Euler-Lagrange equation
\begin{equation} \label{el_rho}
    \int_{B_1(0)} \langle \partial_z H_{\rho}(x, Du_\rho), D\varphi \rangle \, {\rm d}x = 0,
\end{equation}
for all $\varphi \in W^{1,\vm}_0(B_1(0))$.
Let us observe that $z \mapsto \partial_{zz}H_\rho \in C(\R^n)$. By \eqref{h_4} - \eqref{h_lambda} we see that, for every $x,y \in B_1(0)$, for every $z, \xi \in \R^n$, the lagrangian $H_\rho(\cdot)$ satisfies
\begin{equation} \label{holder_h}
    |\partial_z H_\rho (x,z) -  \partial_z H_\rho (y,z)| \leqslant c \left (\rho^\alpha  |x - y|^\alpha(\ell_\omega(z))^{q-1} +\rho^\beta |x - y|^\beta(\ell_\omega(z))^{s-1} \right )  
\end{equation}
and
\begin{equation*}
    \lambda_\rho(x,|z|)|\xi|^2 \leqslant \langle \partial_{zz} H_\rho(x,z) \xi, \xi \rangle, \quad |\partial_{zz} H_\rho(x,z)| \leqslant |\Lambda_{\rho}(x,|z|)|, 
\end{equation*}
where $c \equiv c(n,q,s,||a||_{C^{0,\alpha}}, ||b||_{C^{0,\beta}})$ and $\lambda_\rho(x,|z|)\coloneqq \lambda_{\omega, \sigma_\varepsilon}(x_0 + \rho x,|z|)$, $\Lambda_\rho(x,|z|) \coloneqq \Lambda_{\omega, \sigma_\varepsilon}(x_0 + \rho x, |z|)$. 
\\ \\
{\it Step 3:} The frozen problem.
\\ \\
% Let us consider the following concentric balls
% $$ B_{1/16}(0) \Subset B_{1/4}(0) \Subset B_1(0)  $$ and set
% $$d_1 \coloneqq \min{\dist B_{1/4}(0) \Subset}$$
Let $ \v \in (0,1)$ be a fixed number and let $h \in \R^n \setminus \{ 0 \}$ be a vector such that
\begin{equation} \label{h}
    0 < |h| \leqslant \frac{1}{2^{8/\v}}.
\end{equation}
We take a ball $B_{h}\equiv B_{|h|^\v}(x_c)$, where $x_c \in B_{1/(2+2|h|^\v)}(0)$. For $\Bar{R}>0$ we use the following notation $$B_{\Bar{R}h}=B_{\Bar{R}|h|^{\zeta}}(x_c).$$ 
We observe that $B_{8h}  \Subset B_1(0)$. Now, consider the following  comparison problem
\begin{equation} \label{comparison}
    u_\rho + W^{1, \vm}_0(B_{8h}) \ni w \mapsto \int_{B_{8h}} H^{\iota}_\rho ( Dw) \, {\rm d}x,
\end{equation}
where $H^{\iota}_\rho(\cdot)$ is the frozen integral constructed according to the content of Section \ref{frozen}, i.e. 
$$H^{\iota}_\rho(z) =  \ell_{\omega}(z)\log(1+\ell_{\omega}(z)) + a^{\iota}_\rho (B_{8h})(\ell_{\omega}(z))^q + b^{\iota}_\rho(B_{8h})(\ell_{\omega}(z))^s,$$
where $$a^{\iota}_\rho (B_{8h}) \coloneqq \inf_{x \in B_{8h}}a(x_0 + \rho x) + \sigma_\varepsilon, \quad b^{\iota}_\rho(B_{8h}) \coloneqq \inf_{x \in B_{8h}}b(x_0 + \rho x) + \sigma_\varepsilon.$$
Observing that $w \mapsto \int_{B_{8h}} H^{\iota}_\rho ( Dw) \, {\rm d}x$ is strictly convex, we can apply again the Direct Methods of the Calculus of Variations to get the existence of an unique solution $v \in u_\rho + W^{1, \vm}_0(B_{8h})$ of \eqref{comparison}. The followig Euler-Lagrange equation is satisfied
\begin{equation} \label{euler_comparison}
\int_{B_{8h}} \langle H\ir(Dv) , D\varphi \rangle \, {\rm d}x =0, 
\end{equation}
for all $\varphi \in W^{1, \vm}_0(B_{8h})$.
%and the energy estimate
% \begin{equation*}
%     \int_{B_{8h}} H\ir(Dv) \, {\rm d}x \leqslant  \int_{B_{8h}} H\ir(Du_\rho) \, {\rm d}x
% \end{equation*}
% holds true. 
Now, let us set
\begin{align*}
     \lambda^{\iota}_\rho(|z|) \coloneqq \frac{1}{1 + \ell_{\omega}(z)} & + q\min\{1,q-1\} a^\iota_\rho(B_{8h})(\vl(z))^{q-2} \\ & + s\min\{1,s-1\} b^\iota_\rho(B_{8h})(\vl(z))^{s-2},
\end{align*}
\begin{align*}
    \Lambda^{\iota}_\rho(|z|) & \coloneqq c(n)\left (\frac{\log(1+\vl(z))}{\vl(z)} + \frac{1}{1+\vl(z)} \right ) \\ & \quad + q\max\{1,q-1\} a^\iota_\rho(B_{8h})(\vl(z))^{q-2} \\ & \quad + s\max\{1,s-1\} b^\iota_\rho(B_{8h})(\vl(z))^{s-2}.
\end{align*}
We observe that $ z \mapsto \partial_{zz}H\ir(z) \in C(\R^n)$. Moreover, by \eqref{H_froz_4} and \eqref{H_froz_5} we have
\begin{equation*}
    \lambda\ir(|z|)|\xi|^2 \leqslant \langle \partial_{zz} H\ir(z) \xi, \xi \rangle, \quad 
     |\partial_{zz} H\ir(z)| \leqslant \Lambda\ir(|z|), 
\end{equation*}
\begin{equation*}
    \frac{\Lambda\ir(|z|)}{\lambda\ir(|z|)} \leqslant c(n,q,s)(\log(1+\vl(z))), \quad \text{for all } z \in \R^n : |z|\geqslant e-1,
\end{equation*}
and
\begin{equation} \label{blow1}
    \langle\partial_z H\ir (z_1) -  \partial_z H\ir (z_2), z_1 - z_2 \rangle \geqslant c(n,q,s)\mathcal{V}^2_{\rho, \iota}(z_1, z_2, B_{8h}),
\end{equation}
for all $z, z_1, z_2, \xi \in \R^n$, where 
\begin{align} \label{defV}
    \mathcal{V}^2_{\rho, \iota}(z_1, z_2, B_{8h}) \coloneqq |V_{1,1}(z_1)-V_{1,1}(z_2)|^2 &+ a^\iota_\rho(B_{8h})|V_{\omega,q}(z_1)-V_{\omega,q}(z_2)|^2  \notag \\ & +b^\iota_\rho(B_{8h})|V_{\omega,s}(z_1)-V_{\omega,s}(z_2)|^2.
\end{align}
Therefore, we can apply Theorem \ref{teorema3.1} with $ B_r = B_{8h}$, $u_0 = u_\rho$, $H\ios = H\ir$ and $v\ios = v$. Hence, for all $\delta \in (0,1)$, we get
    \begin{equation*} \label{estima_2v}
        ||Dv||_{L^{\infty}(B_{6h})} \leqslant c\left \{\left [ H_\rho^\iota(||Du_\rho||_{L^{\infty}(B_{8h})})\right]^{\delta} ||Du_\rho||_{L^{\infty}(B_{8h})} + 1 \right \},
    \end{equation*}
with $c \equiv c(n,q,s,||a||_{C^{0,\alpha}}, ||b||_{C^{0,\beta}}, \delta)$.
Let us pick
 $${m} \coloneqq \max \{ ||Du_\rho||_{L^{\infty}(B_{8h})}, e-1 \}.$$
Then, the above $L^{\infty}$-bound, applied with $\delta=\delta_1/\mu$, for some $\delta_1 \in (0,1)$, and \eqref{log} applied with $\epsilon = \mu$, imply
\begin{equation} \label{m}
     ||Dv||_{L^{\infty}(B_{6h})} \leqslant c  {m}^{1+\delta_1},
\end{equation}
with $c$ as above.
\\ \\
{\it Step 4:} We determine a comparison estimate for $u_\rho$ and $v$. More precisely, we prove that for some $\delta_2 \in (0,1/2)$ it holds
\begin{align} \label{comp}
      &  \int_{B_{8h}} \mathcal{V}^2_{\rho,\iota}(Du_\rho, Dv, B_{8h}) \, {\rm d}x \notag \\ & \qquad \leqslant  c M^{1-\frac{\delta_2}{2}}\left \{\rho^\alpha  |h|^{\zeta\alpha}   \int_{B_{8h}} (1 +|Du_\rho|)^{q-1 + \delta_2} \, {\rm d}x \right. \notag \\ & \qquad \quad \left. + \rho^\beta  |h|^{\zeta\beta}  \int_{B_{8h}} (1+|Du_\rho|)^{s-1 + \delta_2} \, {\rm d}x \right \},
\end{align}
 with $c\equiv c(n,q,s,||a||_{C^{0,\alpha}}, ||b||_{C^{0,\beta}}, \delta_2)$, and some constant $M>0$.
\\ \\
Let us start considering a positive constant
$${M} \geqslant \max \left \{ ||\Tilde{E}_{\omega,\sigma}(x_0 + \rho x,Du_\rho)||_{L^{\infty}(B_{8h})}, ||\Tilde{E}_{\omega,\sigma}(x_0 + \rho x,e-1)||_{L^{\infty}(B_{8h})} \right \},$$
then
\begin{align} \label{1}
    a^{\iota}_{\sigma}(B_{8h}) (\vl({m}))^q, b^{\iota}_{\sigma}(B_{8h}) (\vl({m}))^s \leqslant {M}\,,
\end{align}
and
\begin{align} \label{2}
   {m}\leqslant \vl({m}) \leqslant {M}.
\end{align}
We set $\mathcal{V}^2_{\rho,\iota} \equiv \mathcal{V}^2_{\rho,\iota}(Du_\rho, Dv, B_{8h})$, then
\begin{align*}
         \int_{B_{8h}} \mathcal{V}^2_{\rho,\iota} \, {\rm d}x & \mathrel{\eqmathbox{\overset{\mathrm{\eqref{blow1}}}{\leqslant}}} c \int_{B_{8h}} \langle \partial H\ir(Du_\rho) - \partial H\ir(Dv), Du_\rho - Dv \rangle \, {\rm d}x \\ & \mathrel{\eqmathbox{\overset{\mathrm{\eqref{euler_comparison}}}{=}}} c \int_{B_{8h}} \langle \partial
      H\ir(Du_\rho), Du_\rho - Dv \rangle \, {\rm d}x \\ & \mathrel{\eqmathbox{\overset{\mathrm{\eqref{el_rho}}}{=}}} c \int_{B_{8h}} \langle \partial
      H\ir(Du_\rho) -  H_\rho(x, Du_\rho), Du_\rho - Dv \rangle \, {\rm d}x \\ & \mathrel{\eqmathbox{\overset{\mathrm{\eqref{holder_h}}}{\leqslant}}} c \int_{B_{8h}}  \rho^\alpha  |h|^{\zeta\alpha}(\ell_\omega(Du_\rho))^{q-1}(|Du_\rho| +|Dv|) \, {\rm d}x
      \\ & \qquad \qquad + c \int_{B_{8h}} \rho^\beta |h|^{\zeta\beta}(\ell_\omega(Du_\rho))^{s-1}(|Du_\rho| +|Dv|) \, {\rm d}x \\ & \mathrel{\eqmathbox{\eqcolon}} \text{I} + \text{II}.
\end{align*}
We estimate I and II. For the first term we use growth condition from below, minimality of $v$ and \eqref{log} with $\epsilon =\delta_2$, for some $\delta_2 \in (0,1/2)$. Then,
\begin{align*}
 \text{I} & \leqslant  c\rho^\alpha  |h|^{\zeta\alpha} (\ell_\omega({m}))^{q-1}\int_{B_{8h}} H^\iota_\rho(Du_\rho) \, {\rm d}x \\
 &   \leqslant c\rho^\alpha  |h|^{\zeta\alpha} (\ell_\omega({m}))^{q-1}\int_{B_{8h}} \vl(Du_\rho)^{1+\frac{\delta_2}{2}} \, {\rm d}x \\ 
 &   \quad  + c\rho^\alpha  |h|^{\zeta\alpha} (\ell_\omega({m}))^{q-1}\int_{B_{8h}} a\ir(B_{8h})(\vl(|Du_\rho|))^q  \, {\rm d}x \\ & 
 \quad  +c \rho^\alpha  |h|^{\zeta\alpha} (\ell_\omega({m}))^{q-1}\int_{B_{8h}} b\ir(B_{8h})(\vl(|Du_\rho|))^s \, {\rm d}x \\ 
 & \eqcolon \text{I}_1 +  \text{I}_2+  \text{I}_3. 
\end{align*}
Let us estimate the three terms in the right hand side above. By \eqref{2}, we have
\begin{align*}
\text{I}_1 &  = c\rho^\alpha  |h|^{\zeta\alpha}(\ell_\omega({m}))^{q-1}\int_{B_{8h}} (\vl(Du_\rho))^{2-q-\frac{\delta_2}{2}}(\vl(Du_\rho))^{q-1+\delta_2} \, {\rm d}x \\ &  \leqslant  c \rho^\alpha  |h|^{\zeta\alpha}(\ell_\omega({m}))^{1-\frac{\delta_2}{2}}\int_{B_{8h}} (\vl(Du_\rho))^{q-1 + \delta_2} \, {\rm d}x \\ & \leqslant c\rho^\alpha  |h|^{\zeta\alpha}{M}^{1-\frac{\delta_2}{2}} \int_{B_{8h}} (\vl(Du_\rho))^{q-1 + \delta_2} \, {\rm d}x.
\end{align*}
Now let us move on with
\begin{align*}
\text{I}_2 & =  c\rho^\alpha  |h|^{\zeta\alpha}(\ell_\omega({m}))^{q-1} a\ir(B_{8h})\int_{B_{8h}} (\vl(Du_\rho))^{1-\delta_2}(\vl(Du_\rho))^{q-1+\delta_2} \, {\rm d}x \\ &  \leqslant  c\rho^\alpha  |h|^{\zeta\alpha}(a\ir(B_{8h}))^{1-\frac{\delta_2}{q}}(\ell_\omega({m}))^{q(1-\frac{\delta_2}{q})} \int_{B_{8h}} (\vl(|Du_\rho|))^{q-1+\delta_2} \, {\rm d}x   \\ &  \leqslant c\rho^\alpha  |h|^{\zeta\alpha}M^{1-\frac{\delta_2}{q}} \int_{B_{8h}} (\vl(|Du_\rho|))^{q-1+\delta_2} \, {\rm d}x,
\end{align*}
where we used both \eqref{1} and \eqref{2}. In a similar way for $\text{I}_3$ we have
\begin{align*}
    \text{I}_3 & = c\rho^\alpha  |h|^{\zeta\alpha} (\ell_\omega({m}))^{q-1} b\ir(B_{8h})\int_{B_{8h}}(\vl(|Du_\rho|))^{s-q+1-\delta_2}(\vl(|Du_\rho|))^{q-1+\delta_2} \, {\rm d}x \\ & \leqslant c \rho^\alpha  |h|^{\zeta\alpha}(b\ir(B_{8h}))^{1-\frac{\delta_2}{s}}(\vl(m))^{s(1-\frac{\delta_2}{s})}\int_{B_{8h}}(\vl(|Du_\rho|))^{q-1+\delta_2} \, {\rm d}x \\ & \leqslant c\rho^\alpha  |h|^{\zeta\alpha}M^{1-\frac{\delta_2}{s}} \int_{B_{8h}} (\vl(|Du_\rho|))^{q-1+\delta_2} \, {\rm d}x\,.
\end{align*}
 Then, observing that $q,s<2$, we get
\begin{align*}
    \text{I}  \leqslant c \rho^\alpha  |h|^{\zeta\alpha} M^{1-\frac{\delta_2}{2}}  \int_{B_{8h}} (\vl(Du_\rho))^{q-1 + \delta_2} \, {\rm d}x.
\end{align*}
In similar fashion,
\begin{align*}
    \text{II}  \leqslant c \rho^\beta  |h|^{\zeta\beta} M^{1-\frac{\delta_2}{2}}  \int_{B_{8h}} (\vl(Du_\rho))^{s-1 + \delta_2} \, {\rm d}x.
\end{align*}
Hence, we can write
\begin{align*}
  &  \int_{B_{8h}} \mathcal{V}^2_{\rho,\iota}(Du_\rho, Dv, B_{8h}) \, {\rm d}x \\ & \qquad \leqslant  c M^{1-\frac{\delta_2}{2}}\left \{\rho^\alpha  |h|^{\zeta\alpha}   \int_{B_{8h}} (\vl(Du_\rho))^{q-1 + \delta_2} \, {\rm d}x \right. \\ & \quad \qquad \left. + \rho^\beta  |h|^{\zeta\beta}  \int_{B_{8h}} (\vl(Du_\rho))^{s-1 + \delta_2} \, {\rm d}x \right \},
\end{align*}
this, together with the very definition of $\vl$ leads to \eqref{comp}.
\\ \\
{\it Step 5:} We prove that the following fractional Caccioppoli inequality holds
\begin{align} \label{step5}
   & \left ( \fint_{B_{\rho/2}(x_0)} (E_{\omega,\sigma_\varepsilon}(x,|Du_{\omega,\varepsilon}|)-k)_{+}^{2\chi} \, {\rm d}x\right )^{\frac{1}{\chi}}  \notag \\ & \qquad \quad + \rho^{2\gamma -n}[(E_{\omega,\sigma_\varepsilon}(\cdot, |Du_{\omega,\varepsilon}|)-k)_+]^2_{\gamma,2,B_{\rho/2}(x_0)} \notag \\ & \qquad \leqslant M^{\delta_1}\fint_{B_\rho(x_0)} (E_{\omega,\sigma_\varepsilon}(x,|Du_{\omega,\varepsilon} |) -k)_+^2 \, {\rm d}x \notag \\ & \qquad \quad + cM^{2(1-\delta_2)}\rho^{2\alpha}\fint_{B_\rho(x_0)}(1+|Du_{\omega,\varepsilon}|)^{2(q-1+\delta_2)} \, {\rm d}x \notag \\ & \qquad \quad +cM^{2(1-\delta_2)}\rho^{2\beta}\fint_{B_\rho(x_0)}(1+|Du_{\omega,\varepsilon}|)^{2(s-1+\delta_2)} \, {\rm d}x \notag \\ & \qquad \quad + cM^{2+q\delta_1 -\frac{\delta_2}{2}} \rho^\alpha \fint_{B_\rho(x_0)} (1+|Du_{\omega,\varepsilon}|)^{q-1+\delta_2} \, {\rm d}x \notag \\ & \qquad \quad + cM^{2+s\delta_1 -\frac{\delta_2}{2}} \rho^\beta \fint_{B_\rho(x_0)} (1+|Du_{\omega,\varepsilon}|)^{s-1+\delta_2} \, {\rm d}x,
\end{align}
 where $B_\rho(x_0) \Subset B_r$, $\delta_1 \in (0,1), \delta_2 \in (0,1/2)$, $\gamma \in (0,\alpha_0)$, $\alpha_0 \coloneqq \frac{
\min\{\alpha,\beta\}}{\min\{\alpha,\beta\} +2}$, $\chi \coloneqq \frac{n}{n-2\gamma}$ and $c \equiv c(n,q,s,||a||_{C^{0,\alpha}}, ||b||_{C^{0,\beta}},  \delta_1,\delta_2, \gamma) $.
\\ \\
We start observing that
\begin{align*}
&  \int_{B_h} |\tau_h(E_{\rho}(x,|Du_\rho|) -k)_+|^2 \, {\rm d}x\\
& \leqslant c  \int_{B_h} |\tau_h(E\ir(|Dv|)-k)_+|^2 \, {\rm d}x +c\int_{B_{4h}}|E\ir(|Du_\rho|) - E_\rho(x,|Du_\rho|)|^2 \, {\rm d}x \\
& \quad + c\int_{B_{4h}}|E\ir(|Dv|)-E\ir(|Du_\rho|)|^2 \, {\rm d}x \\ &  \eqcolon \overline{\text{I}} + \overline{\text{II}} + \overline{\text{III}}.
\end{align*}
Let us estimate the three terms in the right hand side above. For $\overline{\text{I}}$, we recall the following property of difference quotients
\begin{equation} \label{dq}
    \int_{B_h} |\tau_h (E\ir(|Dv|)-k)_+|^2 \, {\rm d}x \leqslant |h|^{2}\int_{B_{2h}}|D(E\ir(|Dv|)-k)_+|^2 \, {\rm d}x.
\end{equation}
Then, \eqref{dq} and \eqref{bm} applied with $\rho = 4|h|^\zeta$, imply
\begin{eqnarray} \label{tau}
 \overline{\text{I}} \notag & \leqslant &\! c|h|^{2(1-\zeta)}\log(1+\vl(||Dv||_{{L^\infty}(B_{4h})}))\int_{B_{4h}}(E\ir(|Dv|)-k)_+^2 \, {\rm d}x \notag \\* & \! \leqslant & c|h|^{2(1-\zeta)}m^{\delta_1}\int_{B_{4h}}(E\ir(|Dv|)-k)_+^2 \, {\rm d}x,
\end{eqnarray}
where $c\equiv c(n,q,s,||a||_{C^{0,\alpha}}, ||b||_{C^{0,\beta}}, \delta_1)$; for the last inequality we used \eqref{log} with $\epsilon = \delta_1/2$ and \eqref{m}, note that $(1+\delta_1)\delta_1/2 \leqslant \delta_1$. For the second term, we have
\begin{eqnarray} \label{e1}
  \overline{\text{II}} & \overset{\eqref{3.10}}{\leqslant} & c\rho^{2\alpha}|h|^{2\zeta\alpha}\int_{B_{4h}}(\vl(Du_\rho))^{2q} \, {\rm d}x \notag +  c\rho^{2\beta}|h|^{2\zeta\beta}\int_{B_{4h}}(\vl(Du_\rho))^{2s} \, {\rm d}x \notag \\  & \overset{\eqref{2}}{\leqslant}& c\rho^{2\alpha}|h|^{2\zeta \alpha}M^{2(1-\delta_2)} \int_{B_{4h}}(1+|Du_\rho|)^{2(q-1+\delta_2)} \, {\rm d}x \notag \\ &&   + c\rho^{2\beta}|h|^{2\zeta \beta}M^{2(1-\delta_2)} \int_{B_{4h}}(1+|Du_\rho|)^{2(s-1+\delta_2)} \, {\rm d}x.
\end{eqnarray}
Finally, 
\begin{eqnarray*} 
 \overline{\text{III}} &\mathrel{\eqmathbox{\overset{\mathrm{\eqref{in1}}}{\leqslant}}}&  c \int_{B_{4h}} |Du_\rho-Dv|^2 \, {\rm d}x \notag \\ &&  + c(a\ir(B_{8h}))^2 \int_{B_{4h}}\left[\omega^2+|Du_\rho|^2+|Dv|^2 \right ]^{q-1}|Du_\rho - Dv|^2 \, {\rm d}x \notag \\ && + c(b\ir(B_{8h}))^2 \int_{B_{4h}}\left[\omega^2 + |Du_\rho|^2+|Dv|^2 \right ]^{s-1}|Du_\rho - Dv|^2 \, {\rm d}x \notag \\ & \mathrel{\eqmathbox{\overset{\mathrm{\eqref{v},\eqref{m}}}{\leqslant}}}&  c 
m^{1+\delta_1}\int_{B_{4h}}|V_{1,1}(Du_\rho) - V_{1,1}(Dv)|^2 \, {\rm d}x \notag \\ &&  + c(a^{\iota}_\rho(B_{8h}))^2m^{q(1+\delta_1)} \int_{B_{4h}}(\omega^2+|Du_\rho|^2+|Dv|^2 )^{\frac{q-2}{2}}|Du_\rho -Dv|^2 \, {\rm d}x \notag \\ && + c(b^{\iota}_\rho(B_{8h}))^2m^{s(1+\delta_1)} \int_{B_{4h}}(\omega^2+|Du_\rho|^2+|Dv|^2 )^{\frac{s-2}{2}}|Du_\rho -Dv|^2 \, {\rm d}x.
\end{eqnarray*}
We use again \eqref{v} and \eqref{m} to the above content, obtaining
\begin{eqnarray} \label{e2} 
 \overline{\text{III}} & \mathrel{\eqmathbox{\leqslant}}& c 
m^{1+\delta_1}\int_{B_{4h}}|V_{1,1}(Du_\rho) - V_{1,1}(Dv)|^2 \, {\rm d}x \notag \\  &&  + c\left [a^{\iota}_\rho(B_{8h})(\vl(m))^q\right ]m^{q\delta_1} \int_{B_{4h}}a^{\iota}_\rho(B_{8h})|V_{\omega,q}(Du_\rho) - V_{\omega,q}(Dv)|^2 \, {\rm d}x \notag \\  &&  + c \left [b^{\iota}_\rho(B_{8h})(\vl(m))^s \right ]m^{s\delta_1} \int_{B_{4h}}b^{\iota}_\rho(B_{8h})|V_{\omega,s}(Du_\rho) - V_{\omega,s}(Dv)|^2 \, {\rm d}x \notag \\ & \mathrel{\eqmathbox{\overset{\mathrm{\eqref{defV},\eqref{1}}}{\leqslant}}}& c \left (m^{1+\delta_1} + Mm^{q\delta_1}+ Mm^{s\delta_1} \right ) \int_{B_{4h}} \mathcal{V}^2_{\rho,\iota}(Du_\rho, Dv, B_{8h}) \, {\rm d}x \notag \\ & \mathrel{\eqmathbox{\overset{\mathrm{\eqref{2}}}{\leqslant}}}& c\left ( M^{1+q\delta_1}+M^{1+s\delta_1} \right ) \int_{B_{4h}}\mathcal{V}^2_{\rho,\iota}(Du_\rho, Dv, B_{8h}) \, {\rm d}x \notag \\ & \mathrel{\eqmathbox{\overset{\mathrm{\eqref{comp}}}{\leqslant}}}&  c\left ( M^{1+q\delta_1}+M^{1+s\delta_1} \right ) M^{1-\frac{\delta_2}{2}} \notag \\ &&  \times \left \{\rho^\alpha  |h|^{\zeta\alpha}   \int_{B_{4h}} (1 +|Du_\rho|)^{q-1 + \delta_2} \, {\rm d}x + \rho^\beta  |h|^{\zeta\beta}  \int_{B_{4h}} (1+|Du_\rho|)^{s-1 + \delta_2} \, {\rm d}x \right \},\notag
\end{eqnarray}
where $c \equiv c(n,q,s,||a||_{C^{0,\alpha}}, ||b||_{C^{0,\beta}},  \delta_1,\delta_2)$. Now, we choose $\zeta = \frac{2}{\min\{\alpha,\beta\} + 2}$ and we set $\alpha_0 = \frac{\min\{\alpha,\beta\}}{\min\{\alpha,\beta\}+2}$, then
\begin{align} \label{a_0}
     \alpha\zeta, \beta\zeta \leqslant 2\alpha_0  \quad \text{and} \quad 2\alpha_0 = 2(1-\zeta).
\end{align} 
Therefore, by \eqref{2} and \eqref{tau} - \eqref{a_0}, we get
\begin{align} \label{inq}
   & \int_{B_h} |\tau_h(E_{\rho}(x,|Du_\rho|) -k)_+|^2 \, {\rm d}x  \notag  \\ & \qquad \leqslant c|h|^{2\alpha_0}M^{\delta_1}\int_{B_{4h}}(E_\rho(x,|Du_\rho|)-k)_+^2 \, {\rm d}x \notag \\ & \qquad \quad  +c\rho^{2\alpha}|h|^{2\alpha_0}M^{2(1-\delta_2)} \int_{B_{4h}}(1+|Du_\rho|)^{2(q-1+\delta_2)} \, {\rm d}x \notag  \\ & \qquad \quad+ c\rho^{2\beta}|h|^{2\alpha_0}M^{2(1-\delta_2)} \int_{B_{4h}}(1+|Du_\rho|)^{2(s-1+\delta_2)} \, {\rm d}x  \notag \\ & \qquad \quad + c\left ( M^{1+(1+q)\delta_1}+M^{1+(1+s)\delta_1} \right ) M^{1-\frac{\delta_2}{2}} \notag \\ & \qquad  \quad  \times \left \{\rho^\alpha  |h|^{2\alpha_0}   \int_{B_{4h}} (1 +|Du_\rho|)^{q-1 + \delta_2} \, {\rm d}x \right. \notag \\ & \qquad \qquad \left. + \rho^\beta  |h|^{2\alpha_0}  \int_{B_{4h}} (1+|Du_\rho|)^{s-1 + \delta_2} \, {\rm d}x \right \},
\end{align}
with $c \equiv c(n,q,s,||a||_{C^{0,\alpha}}, ||b||_{C^{0,\beta}},  \delta_1,\delta_2)$. We now proceed with a covering argument: we take a lattice of cubes $\{Q_i\}_{i \leqslant \mathcal{I}}$ centered at points $\{ x_i\}_{i \leqslant \mathcal{I}} \subset B_{1/2 + 2|h|^\zeta}(0)$, with sidelength equal to $2|h|^\zeta / \sqrt{n}$, with sides parallel to the coordinate axes, and such that
\begin{equation} \label{balls}
    \Big | B_{1/2}(0) \setminus \bigcup_{i \leqslant \mathcal{I}} Q_i\Big |=0, \quad Q_{i_1} \cap Q_{i_2} = \emptyset \quad \text{if and only if} \quad i_1 \neq i_2.
\end{equation}
Defining $B_i \coloneqq B_{|h|^\zeta}(x_i)$, by construction and by \eqref{h}, we have 
$$ B_{8i} \Subset B_1(0), \quad \text{for all} \quad i \leqslant \mathcal{I} \quad \text{and} \quad \mathcal{I} \sim c(n) \frac{1}{|h|^{n\zeta}}.$$
By \eqref{balls}, each of the dilatated balls $B_{8i_t}$ intersects the similar ones $B_{8i_{\Tilde{t}}}$ fewer than $c_{\mathcal{I}}$ times; therefore the family $\{B_{8i}\}$ has the uniform finite intersection property that implies
\begin{equation} \label{sum}
    \sum_{i=1}^{\mathcal{I}} \lambda(B_{8i}) = c_{\mathcal{I}}\lambda (B_1(0)).
\end{equation}
Now, we sum up inequalities \eqref{inq} for $i \leqslant \mathcal{I}$ and $B_h \equiv B_i$, obtaining
\begin{align*}
   & \int_{B_{1/2}(0)} |\tau_h(E_{\rho}(x,|Du_\rho|) -k)_+|^2 \, {\rm d}x \notag \\ & \qquad \leqslant c|h|^{2\alpha_0}M^{\delta_1} \sum_{i=1}^\mathcal{I} \int_{B_{8i}}(E_\rho(x,|Du_\rho|)-k)_+^2 \, {\rm d}x \notag \\ & \qquad \quad  + c\rho^{2\alpha}|h|^{2\alpha_0}M^{2(1-\delta_2)} \sum_{i=1}^\mathcal{I} \int_{B_{8i}}(1+|Du_\rho|)^{2(q-1+\delta_2)} \, {\rm d}x \notag  \\ & \qquad \quad +  c\rho^{2\beta}|h|^{2\alpha_0}M^{2(1-\delta_2)} \sum_{i=1}^\mathcal{I} \int_{B_{8i}}(1+|Du_\rho|)^{2(s-1+\delta_2)} \, {\rm d}x  \notag \\ & \qquad \quad + c\left ( M^{1+(1+q)\delta_1}+M^{1+(1+s)\delta_1} \right ) M^{1-\frac{\delta_2}{2}} \times \notag \\ & \qquad  \quad  \times \sum_{i=1}^\mathcal{I} \left \{\rho^\alpha  |h|^{2\alpha_0}   \int_{B_{8i}} (1 +|Du_\rho|)^{q-1 + \delta_2} \, {\rm d}x + \rho^\beta  |h|^{2\alpha_0}  \int_{B_{8i}} (1+|Du_\rho|)^{s-1 + \delta_2} \, {\rm d}x \right \}. 
\end{align*}
Using \eqref{sum} to the above display we get
\begin{align*}
 & \int_{B_{1/2}(0)} |\tau_h(E_{\rho}(x,|Du_\rho|) -k)_+|^2 \, {\rm d}x \notag \\ & \qquad \leqslant
c_{\mathcal{I}}c|h|^{2\alpha_0}M^{\delta_1}\int_{B_1(0)}(E_\rho(x,|Du_\rho|)-k)_+^2 \, {\rm d}x \notag \\ & \qquad \quad  + c_{\mathcal{I}}c\rho^{2\alpha}|h|^{2\alpha_0}M^{2(1-\delta_2)} \int_{B_{1}(0)}(1+|Du_\rho|)^{2(q-1+\delta_2)} \, {\rm d}x \notag  \\ & \qquad \quad +  c_{\mathcal{I}}c\rho^{2\beta}|h|^{2\alpha_0}M^{2(1-\delta_2)} \int_{B_{1}(0)}(1+|Du_\rho|)^{2(s-1+\delta_2)} \, {\rm d}x  \notag \\ & \qquad \quad +  c_{\mathcal{I}}c\left ( M^{1+(1+q)\delta_1}+M^{1+(1+s)\delta_1} \right ) M^{1-\frac{\delta_2}{2}} \notag \\ & \qquad  \quad  \times \left \{\rho^\alpha  |h|^{2\alpha_0}   \int_{B_{1}(0)} (1 +|Du_\rho|)^{q-1 + \delta_2} \, {\rm d}x \right. \notag \\ & \qquad \qquad   \left. + \rho^\beta  |h|^{2\alpha_0}  \int_{B_{1}(0)} (1+|Du_\rho|)^{s-1 + \delta_2} \, {\rm d}x \right \}.
\end{align*}
We can now apply Lemma \ref{2.3} and deduce that
$$ (E_\rho(\cdot, |Du_\rho|)-k)_+ \in W^{\gamma, 2}(B_{1/2}(0)),$$ for all $\gamma \in (0,\alpha_0)$. Moreover, we have the following bound
\begin{align*}
   & ||E_\rho(\cdot, |Du_\rho|)-k)_+ ||_{L^{2\chi}(B_{1/2}(0))} + [(E_\rho(\cdot, |Du_\rho|) -k)_+]_{\gamma,2,B_{1/2}(0)} \notag \\ & \qquad \leqslant cM^{\frac{\delta_1}{2}}||E_\rho(\cdot, |Du_\rho|) - k)_+||_{L^2(B_1(0))} \notag \\ & \qquad \quad + cM^{1-\delta_2} \rho^\alpha || 1+ |Du_\rho| \ ||_{L^{2(q-1 + \delta_2)}(B_1(0))}^{q-1 + \delta_2} \\ & \qquad \quad + cM^{1-\delta_2} \rho^\beta  || 1+ |Du_\rho| \ ||_{L^{2(s-1 + \delta_2)}(B_1(0))}^{s-1 + \delta_2} \notag \\ & \qquad \quad +c M^{1+(1+q)\frac{\delta_1}{2} - \frac{\delta_2}{4}} \rho^{\frac{\alpha}{2}}  || 1+ |Du_\rho| \ ||_{L^{q-1 + \delta_2}(B_1(0))}^{(q-1 + \delta_2)/2}  \notag \\ & \qquad \quad + cM^{1+(1+s)\frac{\delta_1}{2} - \frac{\delta_2}{4}} \rho^{\frac{\beta}{2}}  || 1+ |Du_\rho| \ ||_{L^{s-1 + \delta_2}(B_1(0))}^{(s-1 + \delta_2)/2},
\end{align*}
for some $c \equiv c(n,q,s,||a||_{C^{0,\alpha}}, ||b||_{C^{0,\beta}},  \delta_1,\delta_2, \gamma) $, with $\chi = \frac{n}{n-2\gamma}$. Then, scaling back in the previous display, squaring and restoring the original notation, we obtain \eqref{step5}.
\\ \\ 
{\em Step 6:}  In this step we show that the following $L^{\infty}$-bounds holds
\begin{align*}
      &  || \Tilde{E}\ose(\cdot, |Du\oe|)||_{L^{\infty}(B_{d})} \notag \\ & \qquad  \leqslant \frac{c}{(R-d)^{n\theta}}||\Tilde{E}\ose(\cdot, |Du\oe|)||^{\theta}_{L^1(B_R)} + ||Du\oe||^{\theta}_{L^1(B_R)} + c,
\end{align*}
whenever $B_d \Subset B_{R}\subseteq B_r$ are concentric balls, with $c \equiv c(n,q,s,||a||_{C^{0,\alpha}}, ||b||_{C^{0,\beta}})$ and $\theta \equiv \theta(n,q,s,\alpha,\beta) \geqslant 1$.
\\ \\
We start taking $\delta_2$ as
\begin{equation} \label{delta_2}
    \delta_2 \coloneqq \frac{1}{2} \min \left \{ 1+\frac{\alpha}{n} -q, 1+\frac{\beta}{n} -s\right \},
\end{equation}
we observe that by \eqref{q_s} $\delta_2 \in (0,1/2)$. We consider concentric balls $B_d \Subset B_{d_1} \Subset B_{d_2} \Subset B_R \subseteq B_r$ and we set $r_0 \coloneqq (d_2 - d_1)/8$. Then, $B_{2r_0}(x_0) \Subset B_{d_2}$, for every $x_0 \in B_{d_1}$. Without loss of generality, we can assume that
\[
    ||\Tilde{E}\ose(\cdot,| Du\oe|)||_{L^{\infty}(B_d)} \geqslant e-1.
\]
By \eqref{holder}, every point is a Lebesgue point for $|Du\oe|$ and therefore also for $E\ose(\cdot, |Du\oe|)$. Now, by \eqref{step5} used on $B_\rho(x_0) \subset B_{r_0}(x_0)$ and Lemma \ref{revlem} applied on $B_{r_0}(x_0)$ with the following choice:
\begin{align*}
    & v\coloneqq E\ose(\cdot, |Du\oe|), \ 
    M \coloneqq  ||\Tilde{E}\ose(\cdot,| Du\oe|)||_{L^{\infty}(B_{d_2})}, \ \gamma \coloneqq \frac{\alpha_0}{2}, \ k_0 \coloneqq 0 \\ & L_0 \coloneqq M^{\frac{\delta_1}{2}}, \ L_1 \equiv L_2 \coloneqq M^{1-\delta_2}, \ L_3  \coloneqq  M^{1+(1+q)\frac{\delta_1}{2} - \frac{\delta_2}{4}}, \ L_4 \coloneqq M^{1+(1+s)\frac{\delta_1}{2} - \frac{\delta_2}{4}}, \\ &\sigma_1 \coloneqq \alpha, \ \sigma_2 \coloneqq \beta, \ \sigma_3 \coloneqq \frac{\alpha}{2}, \ \sigma_4 \coloneqq \frac{\beta}{2}, \ c_* \equiv c_*(n,q,s,||a||_{C^{0,\alpha}}, ||b||_{C^{0,\beta}},  \delta_1), \\ & f_1 = f_2 = f_3 = f_4
    \coloneqq 1+|Du\oe|, \ \theta_1 = \theta_2 = \theta_3 = \theta_4 \coloneqq 1, \\ & l_1  \coloneqq 2(q-1+\delta_2), \ l_2 \coloneqq 2(s-1+\delta_2), \ l_3 \coloneqq q-1+\delta_2, \ l_4 \coloneqq s-1+\delta_2,
\end{align*}
we have
\begin{align*}
   & || E\ose(\cdot, |Du\oe|)||_{L^{\infty}(B_{d_1})} \\ & \qquad \leqslant \frac{cM^{\frac{\delta_1 \chi}{2(\chi -1)}}}{(d_2 - d_1)^{n/2}} \left ( \int_{B_{d_2}} (E\ose(x,|Du\oe|) )^2  \, {\rm d}x \right )^{\frac{1}{2}} \\ & \qquad \quad + cM^{\frac{\delta_1}{2(\chi -1)} + 1-\delta_2}||P^{2(q-1 + \delta_2),1}_{2,\alpha}(1+|Du\oe|, \cdot, (d_2 - d_1)/4)||_{L^{\infty}(B_{d_1})}  \\ & \qquad \quad + cM^{\frac{\delta_1}{2(\chi -1)} + 1-\delta_2}||P^{2(s-1 + \delta_2),1}_{2,\beta}(1+|Du\oe|, \cdot, (d_2 - d_1)/4)||_{L^{\infty}(B_{d_1})}  \\ & \qquad \quad + cM^{\frac{\delta_1}{2(\chi -1)} + 1 +(1+q)\frac{\delta_1}{2} -\frac{\delta_2}{4}}||P^{q-1 + \delta_2,1}_{2,\alpha/2}(1+|Du\oe|, \cdot, (d_2 - d_1)/4)||_{L^{\infty}(B_{d_1})} \\ & \qquad \quad + cM^{\frac{\delta_1}{2(\chi -1)} + 1 +(1+s)\frac{\delta_1}{2} -\frac{\delta_2}{4}}||P^{s-1 + \delta_2,1}_{2,\beta/2}(1+|Du\oe|, \cdot, (d_2 - d_1)/4)||_{L^{\infty}(B_{d_1})},
\end{align*}
with $c \equiv c(n,q,s,||a||_{C^{0,\alpha}}, ||b||_{C^{0,\beta}},  \delta_1) $; we used the fact that $x_0 \in B_{d_1}$ is an arbitrary point. Now, the inequality in the last display together with \eqref{tilde_e_1} led to
\begin{align} \label{5.55}
    & || \Tilde{E}\ose(\cdot, |Du\oe|)||_{L^{\infty}(B_{d_1})} \notag \\ & \qquad \leqslant \frac{c}{(d_2 - d_1)^{n/2}}||\Tilde{E}\ose(\cdot, |Du\oe|)||^{\frac{\delta_1 \chi}{2(\chi -1)} + \frac{1}{2}}_{L^{\infty}(B_{d_2})} \left ( \int_{B_R}\Tilde{E}\ose(x, |Du\oe|) \, {\rm d}x \right )^{\frac{1}{2}}  \notag \\ & \qquad \quad + c||\Tilde{E}\ose(\cdot, |Du\oe|)||_{L^{\infty}(B_{d_2})}^{\frac{\delta_1}{2(\chi -1)} + 1-\delta_2} \notag  \\ & \qquad \qquad \times ||P^{2(q-1 + \delta_2),1}_{2,\alpha}(1+|Du\oe|, \cdot, (d_2 - d_1)/4)||_{L^{\infty}(B_{d_1})} \notag \\ & \qquad \quad + c||\Tilde{E}\ose(\cdot, |Du\oe|)||_{L^{\infty}(B_{d_2})}^{\frac{\delta_1}{2(\chi -1)} + 1-\delta_2} \notag \\ & \qquad \qquad \times ||P^{2(s-1 + \delta_2),1}_{2,\beta}(1+|Du\oe|, \cdot, (d_2 - d_1)/4)||_{L^{\infty}(B_{d_1})} \notag \\ & \qquad \quad + c||\Tilde{E}\ose(\cdot, |Du\oe|)||_{L^{\infty}(B_{d_2})}^{\frac{\delta_1}{2(\chi -1)} + 1 +(1+q)\frac{\delta_1}{2} -\frac{\delta_2}{4}} \notag \\ & \qquad \qquad \times ||P^{q-1 + \delta_2,1}_{2,\alpha/2}(1+|Du\oe|, \cdot, (d_2 - d_1)/4)||_{L^{\infty}(B_{d_1})} \notag \\ & \qquad \quad + c||\Tilde{E}\ose(\cdot, |Du\oe|)||^{\frac{\delta_1}{2(\chi -1)}+ 1 +(1+s)\frac{\delta_1}{2} -\frac{\delta_2}{4}}_{L^{\infty}(B_{d_2})} \notag \\ & \qquad \qquad \times||P^{s-1 + \delta_2,1}_{2,\beta/2}(1+|Du\oe|, \cdot, (d_2 - d_1)/4)||_{L^{\infty}(B_{d_1})} + c.
\end{align}
Let us choose $\delta_1 \in (0,1)$ such that
\begin{equation} \label{delta_1}
    \begin{cases}
       &  \frac{\delta_1 \chi}{2(\chi -1)} + \frac{1}{2} <1 \\ & \frac{\delta_1}{2(\chi -1)} + 1-\delta_2 < 1 \\ &  \frac{\delta_1}{2(\chi -1)} + 1 +(1+q)\frac{\delta_1}{2} -\frac{\delta_2}{4} <  1 \\ & \frac{\delta_1}{2(\chi -1)}+ 1 +(1+s)\frac{\delta_1}{2} -\frac{\delta_2}{4} < 1.
    \end{cases}
\end{equation}
From now on, without loss of generality, we assume that
\begin{align*} 
    n>2\alpha, \quad n>2\beta,
\end{align*}
indeed, since we have strict inequalities in \eqref{q_s}, we can take $\alpha, \beta < 1$.
By the choice of $\delta_2$ in \eqref{delta_2} we have
$$ \max \left \{ \frac{n(q-1+\delta_2)}{\alpha}, \frac{n(s-1+\delta_2)}{\beta}\right \} < 1,$$
so we can apply Lemma \ref{lem2.4} with $\gamma =1$ and obtain
\begin{equation*}
    \begin{cases}
       & ||P^{q-1+\delta_2,1}_{2,\alpha/2}(1+|Du\oe|, \cdot, (d_2 - d_1)/4)||_{L^\infty(B_{d_1})} \\ & \qquad \leqslant c ||1+ |Du\oe| \ ||_{L^1(B_R)}^{(q-1+\delta_2)/2} \\ &  ||P^{s-1+\delta_2,1}_{2,\beta/2}(1+|Du\oe|, \cdot, (d_2 - d_1)/4)||_{L^\infty(B_{d_1})} \\ & \qquad \leqslant c ||1+ |Du\oe| \ ||_{L^1(B_R)}^{(s-1+\delta_2)/2} \\ & ||P^{2(q-1+\delta_2),1}_{2,\alpha}(1+|Du\oe|, \cdot, (d_2 - d_1)/4)||_{L^\infty(B_{d_1})} \\ & \qquad \leqslant c ||1+ |Du\oe| \ ||_{L^1(B_R)}^{q-1+\delta_2}  \\ & ||P^{2(s-1+\delta_2),1}_{2,\beta}(1+|Du\oe|, \cdot, (d_2 - d_1)/4)||_{L^\infty(B_{d_1})} \\ & \qquad \leqslant c ||1+ |Du\oe| \ ||_{L^1(B_R)}^{s-1+\delta_2},
    \end{cases}
\end{equation*}
with $c \equiv c(n,q,s,\alpha,\beta)$. Then, \eqref{5.55} becomes 
\begin{align*}
& || \Tilde{E}\ose(\cdot, |Du\oe|)||_{L^{\infty}(B_{d_1})} \notag \\ & \qquad \leqslant \frac{c}{(d_2 - d_1)^{n/2}}||\Tilde{E}\ose(\cdot, |Du\oe|)||^{\frac{\delta_1 \chi}{2(\chi -1)} + \frac{1}{2}}_{L^{\infty}(B_{d_2})} \left ( \int_{B_R}\Tilde{E}\ose(x, |Du\oe|) \, {\rm d}x \right )^{\frac{1}{2}}  \notag \\ & \qquad \quad + c||\Tilde{E}\ose(\cdot, |Du\oe|)||_{L^{\infty}(B_{d_2})}^{\frac{\delta_1}{2(\chi -1)} + 1-\delta_2}  ||1+ |Du\oe| \ ||_{L^1(B_R)}^{q-1+\delta_2} \notag \\ & \qquad \quad + c||\Tilde{E}\ose(\cdot, |Du\oe|)||_{L^{\infty}(B_{d_2})}^{\frac{\delta_1}{2(\chi -1)} + 1-\delta_2}  ||1+ |Du\oe| \ ||_{L^1(B_R)}^{s-1+\delta_2} \notag \\ & \qquad \quad + c||\Tilde{E}\ose(\cdot, |Du\oe|)||_{L^{\infty}(B_{d_2})}^{\frac{\delta_1}{2(\chi -1)}+ 1 +(1+q)\frac{\delta_1}{2} -\frac{\delta_2}{4}}  ||1+ |Du\oe| \ ||_{L^1(B_R)}^{q-1+\delta_2} \notag \\ & \qquad \quad + c||\Tilde{E}\ose(\cdot, |Du\oe|)||_{L^{\infty}(B_{d_2})}^{\frac{\delta_1}{2(\chi -1)}+ 1 +(1+s)\frac{\delta_1}{2} -\frac{\delta_2}{4}}  ||1+ |Du\oe| \ ||_{L^1(B_R)}^{s-1+\delta_2} + c,
\end{align*}
with $c \equiv c(n,q,s,||a||_{C^{0,\alpha}}, ||b||_{C^{0,\beta}}) $. By \eqref{delta_1} we can apply Young's inequality to the above diplayed inequality, getting
\begin{align*}
  &  || \Tilde{E}\ose(\cdot, |Du\oe|)||_{L^{\infty}(B_{d_1})} \leqslant \frac{1}{2}||\Tilde{E}\ose(\cdot, |Du\oe|)||_{L^{\infty}(B_{d_2})} \\ & \qquad  +\frac{c}{(d_2-d_1)^{n\theta}}||\Tilde{E}\ose(\cdot, |Du\oe|)||^{\theta}_{L^1(B_R)} + ||Du\oe||^{\theta}_{L^1(B_R)} + c.
\end{align*}
At this point we use Lemma \ref{lemma_giusti} to obtain
\begin{align*}
      &  || \Tilde{E}\ose(\cdot, |Du\oe|)||_{L^{\infty}(B_{d})} \notag \\ & \qquad  \leqslant \frac{c}{(R-d)^{n\theta}}||\Tilde{E}\ose(\cdot, |Du\oe|)||^{\theta}_{L^1(B_R)} + ||Du\oe||^{\theta}_{L^1(B_R)} + c,
\end{align*}
where $c \equiv c(n,q,s,||a||_{C^{0,\alpha}}, ||b||_{C^{0,\beta}})$ and $\theta \equiv \theta(n,q,s,\alpha,\beta) \geqslant 1$. 
\\ \\
{\it Step 7:} Convergence
\\ \\
We recall that, by very definition \eqref{tmq}
$$\Tilde{E}\ose(x,|z|) \leqslant 2(1+H\ose(x,z)), \quad \text{for all } (x,z) \in B_r \times \R^n.$$
Therefore, using the information in the last display together with \eqref{h_1} and \eqref{con_l}, we obtain 
\begin{align*}
 || \Tilde{E}\ose(\cdot, |Du\oe|)||_{L^{\infty}(B_{d})} & \leqslant \frac{c}{(R-d)^{n\theta}} \left ( \int_{B_R} (1+H\ose(x,Du\oe)) \, {\rm d}x\right )^{\theta} + c \\ & \leqslant \frac{c}{(R-d)^{n\theta}} \left ( \mathcal{L}(u,B_R) + o_\varepsilon(\omega) + o(\varepsilon) +|B_R|\right)^{\theta} + c,
\end{align*}
with $c \equiv c(n,q,s,||a||_{C^{0,\alpha}}, ||b||_{C^{0,\beta}})$ and $\theta \equiv \theta(n,q,s,\alpha,\beta) \geqslant 1$. Now, by \eqref{tmq} we have
\begin{align} \label{ul}
||Du\oe||_{L^{\infty}(B_d)} \leqslant \frac{c}{(R-d)^{n\theta}} \left ( \mathcal{L}(u,B_R) + o_\varepsilon(\omega) + o(\varepsilon) +|B_R|\right)^{\theta} + c,
\end{align}
this means that, up to not relabelled subsequences, by \eqref{con}, 
$$ u\oe \rightharpoonup^* u_\varepsilon \quad \text{in } W^{1,\infty}(B_d),$$
for every $\varepsilon >0$. Therefore, letting $\omega \to 0$,  \eqref{ul} implies
\begin{equation*}
  ||Du_\varepsilon||_{L^{\infty}(B_d)} \leqslant \frac{c}{(R-d)^{n\theta}} \left ( \mathcal{L}(u,B_R) + o(\varepsilon) +|B_R|\right)^{\theta} + c.
\end{equation*}
We now recall \eqref{con2} to get
$$ u_\varepsilon \rightharpoonup^* u \quad \text{in } W^{1,\infty}(B_d).$$
Hence, we can conclude that
\begin{align*}
  ||Du||_{L^{\infty}(B_d)} \leqslant \frac{c}{(R-d)^{n\theta}} \left ( \mathcal{L}(u,B_R) +|B_R|\right)^{\theta} + c,
\end{align*}
with $c$ and $\theta$ as above. We note that the choices $d=r/2$ and $R=r$ are admissible, so \eqref{l_infty_estimate} follows. \\ \\
{\em Step 8:} $Du$ is locally H\"older continuous.
\\ \\ 
The gradient local Hölder continuity now follows by a similar argument as the one used in Sections 5.9 and 5.10 of \cite{DFM23}. 
\\ \indent
The proof is complete.
\subsection*{Proof of Corollary \ref{maincor}} Corollary \ref{maincor} directly comes from repeated applications of the arguments detailed in the proof of Theorem \ref{theorem4.1}.

\begin{rem}
We observe that we have analyzed the multidimensional scalar case. But, since the integrand has the Uhlenbeck structure
$$H(x,z) = \tilde{H}(x,|z|),$$
it is possible to extend the result to functions $u : \Omega \to \R^N$, with $N \geqslant 2$.
\end{rem}

\end{document}